\documentclass[12pt,a4paper]{article}

\usepackage{graphicx}
\usepackage{latexsym,amssymb,amsmath}
\usepackage{amsfonts, mathrsfs}
\usepackage{latexsym}
\usepackage{graphicx}
\usepackage{color}
\usepackage{amsthm}
\usepackage{enumitem}
\usepackage[all]{xy}

\theoremstyle{plain}
\newtheorem{theorem}{Theorem}[section]
\newtheorem{proposition}{Proposition}[section]
\newtheorem{corollary}{Corollary}[section]
\newtheorem{lemma}{Lemma}[section]

\theoremstyle{definition}
\newtheorem{definition}{Definition}[section]
\newtheorem{remark}{Remark}[section]
\newtheorem{example}{Example}[section]

\def\limfunc#1{#1}

\numberwithin{equation}{section}
\normalsize
\makeatletter
\renewcommand\section{\@startsection{section}{1}{\z@}%
                                   {-5.5ex \@plus -1ex \@minus -.2ex}%
                                   {1.3ex \@plus.2ex}%
                                   {\normalfont\normalsize\bfseries}}
\renewcommand{\@seccntformat}[1]{\csname the#1\endcsname.\ }

\renewcommand\subsection{\@startsection{subsection}{2}{\z@}%
                                   {-2.5ex \@plus -1ex \@minus -.2ex}%
                                   {1.3ex \@plus.2ex}%
                                   {\normalfont\normalsize\it}}
\makeatother

\DeclareMathOperator{\epi}{epi}

\DeclareMathOperator{\dom}{dom}
\DeclareMathOperator{\inte}{int}
\DeclareMathOperator{\co}{co}

\DeclareMathOperator{\sco}{sco}
\DeclareMathOperator{\scl}{scl}
\DeclareMathOperator{\sclco}{sclco}
\DeclareMathOperator{\cl}{cl}

\DeclareMathOperator{\supp}{supp}
\DeclareMathOperator{\winf}{WInf}
\DeclareMathOperator{\wsup}{WSup}
\DeclareMathOperator{\wmin}{WMin}
\DeclareMathOperator{\wmax}{WMax}

\def\RR{\mathbb{R}}
\def\A{\mathcal{A}}
\def\B{\mathcal{B}}
\def\L{\mathcal{L}}
\def\U{\mathcal{U}}
\def\V{\mathcal{V}}
\def\W{\mathcal{W}}
\def\R{\mathbb{R}}

\def\la{\langle}
\def\ra{\rangle}

\begin{document}

\setlength{\abovedisplayskip}{3pt}
\setlength{\belowdisplayskip}{3pt}

%\title{\large\bf Strong Robust Duality for Robust Convex Vector Optimization Problems via Robust Convex Vector Farkas Lemmas} 
% STRONG ROBUST DUALITY FOR  ROBUST CONVEX VECTOR OPTIMIZATION PROBLEMS}

\title{\large\bf  Sectional convexity of  epigraphs of conjugate mappings with applications to    robust vector  duality 
}

\author{{ \normalsize N.~Dinh}\thanks{
International University, Vietnam National University - HCMC, Linh Trung
ward, Thu Duc district, Ho Chi Minh city, Vietnam ({\it ndinh@hcmiu.edu.vn}). Parts of the work of this author is supported by the NAFOSTED, Vietnam. } 
\and {\normalsize D.H.~Long}\thanks{VNUHCM - University of Science, District 5, Ho Chi Minh city, Vietnam, and  Tien Giang University, Tien Giang town, Vietnam  ({\it danghailong@tgu.edu.vn}).}
 }

\maketitle

\centerline{\it  Dedicated to Professor Hoang Tuy's $90^{th}$ birthday} 
\begin{abstract}
This paper concerns the robust vector problems 
\begin{equation*}
\begin{array}{rrl}
\mathrm{(RVP)}\ \  & \wmin\left\{ F(x): x\in C,\; G_u(x)\in
-S,\;\forall u\in\mathcal{U}\right\} , & 
\end{array}
\label{rvop}
\end{equation*}
where    $X, Y, Z$ are locally convex Hausdorff topological vector spaces, $K$ is a  closed and convex cone in $Y$ with nonempty interior,  
and $S$ is a closed, convex cone in $Z$,
$\mathcal{U}$ is an \textit{uncertainty set}, $F\colon X\rightarrow {Y}^\bullet,$ $G_u\colon X\rightarrow Z^\bullet$ are proper mappings for all $ u \in \U$,  and  $\emptyset \ne C\subset X$.  Let $ 
A:=C\cap \left(\bigcap_{u\in\mathcal{U}}G_u^{-1}(-S)\right)$ and $I_A : X \to Y^\bullet  $ be the indicator  map defined by $I_A(x) = 0_Y $ if  $x \in A$ and $I_A(x) = + \infty_Y$ if $ x \not\in A$.

It is well-known that the epigraph of the conjugate mapping  $(F+I_A)^\ast$, in general,   is not a convex set. We show that, however, it is ``$k$-sectionally convex" in the sense  that  each section form by the intersection of $\epi (F+I_A)^\ast$  and   any translation of a ``specific $k$-direction-subspace" is a convex subset, 
for any $k$ taking from $\inte K$. 

The key results of the paper are the representations of the epigraph of the conjugate mapping  $(F+I_A)^\ast$  via the     closure of the  $k$-sectionally convex  hull of a union of epigraphs  of conjugate mappings of   mappings from a family  involving the data  of the problem  (RVP).   The  results are then given rise to stable  robust vector/convex vector  Farkas lemmas which, in turn,   are used to establish new results on   robust strong  stable duality results for (RVP). It is shown at the end of the paper that,  when specifying the result to some concrete classes of scalar robust problems (i.e.,   when $Y = \R$), our results cover and extend  several corresponding known ones  in the literature. 
 \end{abstract}



\section{Introduction}

Let $X, Y, Z$ be locally convex Hausdorff topological vector spaces (briefly, lcHtvs)  with topological dual spaces denoted  by $X^\ast, Y^\ast, Z^\ast$, respectively.  
The only topology we consider on dual spaces is the weak*-topology. For a set $U\subset X$, we denote by $\cl U$, $\inte U$, $\co U$, and $\cl \co U$ the \textit{closure}, the {\it interior},  the \textit{convex hull}, and the \textit{closed and convex hull} of $U$, respectively. Note that $\cl\co U=\cl(\co U)$.

We consider the \textit{robust vector optimization problem} of the model \cite{Robust}, \cite{DL17VJM}:   
\begin{equation}
\begin{array}{rrl}
\mathrm{(RVP)}\ \  & \wmin\left\{ F(x): x\in C,\; G_u(x)\in
-S,\;\forall u\in\mathcal{U}\right\} , & 
\end{array}
\label{rvop}
\end{equation}
where $K$ is a closed and convex cone in $Y$ with nonempty interior,  
and $S$ is a closed, convex cone in $Z$,
$\mathcal{U}$ is an \textit{uncertainty set}, $F\colon X\rightarrow {Y}^\bullet,$ $G_u\colon X\rightarrow Z^\bullet$ are proper mappings,  and  $\emptyset \ne C\subset X$.
The feasible set of (RVP) is 
\begin{equation} \label{feas-A}
A:=C\cap \left(\bigcap_{u\in\mathcal{U}}G_u^{-1}(-S)\right).\end{equation} 
We assume through out this paper that 
$A\cap\dom F\ne \emptyset$. 

Robust optimization provides a way to approach   optimization problems with  uncertain  data.     The subject  has attracted  
attention of  many mathematicians around the world  in the last decades (see \cite{Barro},   \cite{BB9},    \cite{BEN09} and the comprehensive survey papers \cite{Ber_survey} and \cite{Gabrel}). Many works  devoted to the duality results on  (both scalar and vector) robust optimization problems appeared in the literature
(see \cite{Bot10}, \cite{2}, \cite{DMVV17-Robust-SIAM}, \cite{DGLV17-Robust}, \cite{GJLL13},   \cite{JL10}, \cite{JLW13}, \cite{Tan92},     and the references therein).  In the recent years, the extension of Farkas-type results  to systems with vector-valued functions, enables the authors  to study the duality for vector optimization problems   \cite {DGLM16},\cite{epi}, and also, the duality  for robust vector optimization problems  \cite{Robust},   \cite{DL17VJM}.     

Motivated by the works \cite{Robust},  \cite{DL17VJM}, the present paper continues the study on  the duality  for robust vector problem of the models (RVP). The key results of the paper are the representations of the epigraph of the conjugate mapping  $(F+I_A)^\ast$ via    closure of the  sectionally convex  hull of a union of epigraphs  of conjugate mappings of   mappings from a family  involving the data  of the problem  (RVP). With this better understanding on the presentation of $\epi (F + I_A)^\ast$, we propose several new  concrete qualification conditions,  establish   stable  robust vector/convex vector  Farkas lemmas. These results are then used   used to establish new results on  robust strong stable  duality results for (RVP).

The paper is organized as follows: In Section 2, we  recall some notations on weak supremum and weak infimum of a set in locally topological vector spaces ordered   by  closed convex cones with non-empty interiors, and some important properties of these notions.    In Section 3  we introduce  some basic mathematical tools which will be the key tools for the our study. 
 We first introduce the notions of ``uniformly $S^+$-concave of a family mappings and ``$S^+$-uniform usc of a mapping. Notions of    $k$-sectional convexity, $k$-sectional closedness, ``$k$-sectional convex hull"  of  sets in a product  space are then introduced together with some important properties of these notions.   Section 4 is used the establishing  various representations of the epigraph of the conjugate mapping $(F +I_A)^\ast$,  $\epi (F + I_A)^\ast$ via closure of  sectional convex hull of a union of epigraphs of conjugate mappings involving in the problem (RVP).  These results  then  play  crucial roles  in proving   the main results  sections that follows.  Stable  robust vector Farkas lemma are established in   Section 5. Section 6 is left for duality results for robust vector problems (RVP). It is shown in Section 7 that   when specifying the results obtained in Section 6  to some concrete classes of scalar robust problems (i.e.,   when $Y = \R$), our results still cover and extend  several corresponding known ones  in the literature.

\section{Preliminaries and notations}

We now recall some notions and properties that will be useful in the sequent.

Let $K\subsetneqq Y $ be a   closed and convex cone in $Y$ with nonempty interior\footnote{It is also called  a proper cone.}, i.e.,  $\inte K\ne \emptyset$. It is worth observing that $K+\inte K=\inte K$, and consequently,
\begin{equation}\label{eq_1aaa}
\left.\begin{array}{c}
y\in K\\
y' \notin  -\inte K
\end{array}
\right\} \ \Longrightarrow \ y+y'\notin -\inte K.
\end{equation}
The next properties are useful in the sequel.

\begin{lemma}\label{pro_2.1aaa}
Let $y\in Y$, {$k_0\in K$}, $\lambda\in ]0,1[$ and $\alpha,\beta\in \RR$. Then
\begin{equation}\label{eq_1bis}
\left.\begin{array}{c}
y+\alpha k_0 \notin -\inte K\\
y+\beta k_0\notin -\inte K
\end{array}
\right\} \ \Longrightarrow \ y+[\lambda\alpha+(1-\lambda )\beta]k_0\notin -\inte K.
\end{equation}
\end{lemma}

\begin{proof}
Let us assume that  the statement in the right-hand side of \eqref{eq_1bis} is false, i.e.,
\begin{equation}\label{cts6}
y+[\lambda\alpha+(1-\lambda )\beta]k_0\in -\inte K.
\end{equation}
If  $\alpha\le \beta$ then $\alpha-[\lambda \alpha+(1-\lambda)\beta] =(1-\lambda)(\alpha-\beta)\le 0$, 
which, together with the fact that  {$k_0\in  K$},  leads  to  
\begin{equation}\label{cts7-l}
\left\{ \alpha  -[\lambda \alpha+(1-\lambda)\beta]\right\}  k_0    =   \alpha k_0 -[\lambda \alpha+(1-\lambda)\beta] k_0    \in -K. 
\end{equation}
From \eqref{cts6} and \eqref{cts7-l}, we get 
$y+\alpha k_0\in -\inte K-K= -\inte K$. 

In case $\beta \leq \alpha$, one has  $\beta -[\lambda \alpha+(1-\lambda)\beta] =\lambda    (\beta - \alpha )\le 0$. The same argument as above  leads to 
$y+ \beta  k_0\in-\inte K$.   Consequently, \eqref{eq_1bis} holds. 
\end{proof}

\begin{lemma}\label{lem_2.2zba}
Let $y\in Y$ and $k_0\in \inte K$. Then, there exist $\bar\alpha\in \mathbb{R}$ such that for $\alpha \in \R$,  it holds
\begin{equation}
\label{eq_2.5zba}
\alpha<\bar \alpha  \  \; \Longleftrightarrow\; \ \big(  y+\alpha k_0\in -\inte K \big).
\end{equation}
\end{lemma}

\begin{proof}
Let denote $P:=\{\alpha\in \mathbb{R}: y+\alpha k_0\in -\inte K\}$.

$\bullet$ Observe firstly that  $P\ne\emptyset$ and $P \ne \R$. 
Indeed, as $k_0\in \inte K$, there exist balanced, convex  neighborhoods $V_1$ and $V_2$ of $0_Y$ such that $k_0 + V_1 \in \inte K$ and $-k_0+V_2\subset -\inte K$, respectively.  Taking $\lambda_2>0$ with  $\lambda_2  y\in V_2$, one has   $ - k_0 + \lambda_2 y \in - \inte K$ or equivalently, $y-\frac{1}{\lambda_2}k_0\in -\inte K$, that means   $-\frac{1}{\lambda_2}\in P$, showing  that $P\ne \emptyset$. Now, taking $\lambda_1 >0$ such that $\lambda_1  y\in V_1$. We then have 
 $ y + \lambda_1 k_0 \in \inte K$ or equivalently,  $y+\frac{1}{\lambda_1}k_0\in \inte K$. 
As  $K$ is a proper cone, 
  $(-\inte K)\cap (\inte K)=\emptyset$, and hence,  $y+\frac{1}{\lambda_1}k_0\notin -\inte K$, or equivalently, $\frac{1}{\lambda_1}\notin P$ yielding $P\ne \mathbb{R}$. It is also worth mentioning that for any $ \lambda > 0$ such that  $\lambda < \lambda_1$ then  $\lambda y \in V_1$ and, by the same argument as above  we have $ \frac{1}{\lambda} \notin P$, meaning that $P$ is bounded above.

So, one has  $\bar \alpha:=\sup P\in \mathbb{R}$. We now prove that \eqref{eq_2.5zba} holds.

$\bullet$  $[\Longleftarrow]$ Assume that $y+\alpha k_0\in -\inte K$. Then, $\alpha \in P$ and hence,  $\alpha\le   \bar \alpha =  \sup P$. 
Now, if  $\alpha=\bar \alpha$ then, as $y+\alpha k_0\in -\inte K$, there is a neighborhood $V$ of $0_Y$ such that $y+\alpha k_0+V\subset -\inte K$. Take $\epsilon>0$ such that $\epsilon k_0\in V$, one has $y+\alpha k_0 +\epsilon k_0 = y+(\alpha+\epsilon)k_0\in -\inte K$ which yields $\alpha+\epsilon \in P$. We then have  $\alpha+\epsilon >\alpha=\bar\alpha=\sup P$, a contradiction. Consequently,  $\alpha<\bar \alpha$ as expected. 

$\bullet$  $[\Longrightarrow]$  Assume that $\alpha<\bar \alpha$. Then, as $\bar \alpha=\sup P$, there is $\alpha_1\in P$ such that $\alpha<\alpha_1$. As $\alpha_1\in P$, it holds $y+\alpha_1 k_0\in -\inte K$, and hence,
$$y+\alpha k_0=(y+\alpha_1 k_0) +(\alpha-\alpha_1)k_0\in -\inte K-\inte K=-\inte K.$$
The proof is complete.
\end{proof}

In this paper, we shall use the  two orderings generated by the cone $K$: the weak ordering and the usual ordering, defined respectively by, for any  $y_1, y_2 \in Y$, 
\begin{itemize}[wide, nosep]
\item {\it weak ordering:} \   $y_1<_K y_2$ if and only if  $y_1-y_2\in-\inte K $, 

\item {\it usual ordering:}\  $y_1\leqq_K y_2$ if and only if $y_1-y_2\in -K$.
\end{itemize}
We enlarge $Y$ by attaching a \textit{greatest element} $+\infty_Y$ and a \textit{smallest element} $-\infty_Y$  which do not belong to $Y$, and we denote $Y^\bullet:=Y\cup\{-\infty_Y,+\infty_Y\}$. By convention, $-\infty_Y<_K y<_K+\infty_Y$ for any $y\in Y$. We also assume by convention that
\begin{gather}
-(+\infty_Y)=-\infty_Y,\qquad\qquad -(-\infty_Y)=+\infty_Y,\notag\\
(+\infty_Y)+y=y+(+\infty_Y)=+\infty_Y,\quad \forall y\in Y\cup\{+\infty_Y\}\label{conv}\\
(-\infty_Y)+y=y+(-\infty_Y)=-\infty_Y,\quad \forall y\in Y\cup\{-\infty_Y\}\notag
\end{gather}

The sums $(-\infty_Y)+(+\infty_Y)$ and $(+\infty_Y)+(-\infty_Y)$ are not considered in this paper.

\bigskip

Given $\emptyset\ne M\subset Y^\bullet$, the following notions  quoted from \cite[Definition 7.4.1]{2}) will be used throughout this paper.
\begin{itemize}[nosep, wide]
		\item An element $\bar v\in Y^\bullet$ is said to be a \textit{weakly infimal element} of $M$ 
		if for all $v\in M$ we have $v\not<_K \bar v$ and if for any $\tilde{v}\in Y^\bullet$ such that 
		$\bar v<_K\tilde{v}$, then there exists some $v\in M$ satisfying $v<_K\tilde{v}$. The set of all weakly
		infimal elements of $M$ is denoted by $\winf M$ and is called the \textit{weak infimum} of $M$.

		\item An element $\bar v\in Y^\bullet$ is said to be a \textit{weakly supremal element} of $M$
		if for all $v\in M$ we have $\bar v\not<_K  v$ and if for any $\tilde{v}\in Y^\bullet$ such that
		$\tilde v<_K\bar v$, then there exists some $v\in M$ satisfying $\tilde v<_K v$. The set of all weakly supremal elements of $M$ is denoted by $\wsup M$ and is called the \textit{weak supremum} of $M$.

		\item The \textit{weak minimum} of $M$ is the set $\wmin M=M\cap \winf M$ and its 
		elements are the \textit{weakly minimal elements} of $M$. The \textit{weak maximum} of $M$,  
		$\wmax M$,   is defined similarly, $\wmax M :=M\cap \wsup M$.  
	\end{itemize}

Weak infimum and weak supremum of the empty set is defined by convention as $\wsup \emptyset = \{-\infty_Y\}$ and $\winf \emptyset = \{+\infty_Y\}$, 
respectively.

It follows  from the definition of $\wsup M$ that  $\wsup (M+a)=a+\wsup M$ for all $M\subset Y^\bullet$ and $a\in Y$,  and 
\begin{eqnarray*}
+\infty _{Y}\in \wsup M\; &\Longleftrightarrow &\;\wsup M=\{+\infty _{Y}\} 
\notag  \label{wsup=infty*} \\
&\Longleftrightarrow &\;\forall \tilde{v}\in Y,\;\exists v\in M:\tilde{v}%
<_{K}v. 
\end{eqnarray*}
 Next  properties of sets of weak suprema and weak minima  were quoted from \cite[Proposition 2.1]{DGLM16}, \cite[Proposition 7.4.3]{2}, and  \cite[Proposition 2.4]{Tan92}.

 \begin{lemma}\label{lemextra2}  Assume that  $\emptyset \neq M\subset Y$.  Then it holds: 
\begin{itemize}[nosep]
\item[{\rm (a)}] $ \wsup M=\cl(M-\inte K)\setminus (M-\inte K)$,

\item[{\rm (b)}]  If\  $\wsup M\neq \{+\infty
_{Y}\}$   then  $\wsup M-\inte K=M-\inte K$,
\item[{\rm (c)}] If\  $\wsup M\neq \{+\infty
_{Y}\}$   then the following  decomposition of $Y$ holds: 
\begin{equation} \label{Ydecomposition}
Y=\left( M-\inte K\right) \cup (\wsup M)\cup (\wsup M+\inte K), 
\end{equation}
\item[{\rm\ \ \ (d)} ]  If  $\emptyset \ne M,N\subset Y^\bullet$,  then 
\begin{equation} \label{eq_2aa}
 \wsup (\wsup M+\wsup N)= \wsup  (M+\wsup N)=\wsup (M+N).
\end{equation}
\end{itemize}
\end{lemma}

It is worth noting that $\winf M=-\wsup(-M)$ for all $M\subset Y^\bullet$. So, 
all the assertions in Lemma \ref{lemextra2} still hold when all the terms  $\wsup$, $\wmax$, $\inte K$, and $+\infty_Y$  are replaced by  $\winf$, $\wmin$, $-\inte K$, and $-\infty_Y$, respectively.

We denote by $\mathcal{L}(X,Y)$ the space of linear continuous mappings from 
$X$ to $Y$, and by $0_\L$ the zero element  of $\mathcal{L}(X,Y)$ (i.e.,  $0\mathbf{{_{\mathcal{L}}}}(x)=0_{Y}$ for all $x\in X$).
The topology   considered in $\mathcal{L}(X,Y)$ is the one defined by the point-wise convergence, i.e., $(L_\alpha)_{\alpha \in D} \subset \mathcal{L}(X,Y)$ and $L \in \mathcal{L}(X, Y)$,  $L_\alpha\to L$ means that  $L_\alpha(x)\to L(x)$ in $Y$ for all $x\in X$. 

Given a vector-valued mapping $F\colon X\to Y^\bullet$, the \textit{effective domain} of $F$ is defined by
$\dom F:=\{x\in X: F (x)\ne +\infty_Y\}$, and the \textit{$K$-epigraph} of $F$ is defined by
$\epi\nolimits_K F =\{(x,y)\in X\times Y:y\in F(x)+K\}$.  As $K$ is fixed for the   whole  paper, we will  write $\epi F$ instead of $\epi_K F$.
We say that
\begin{itemize}[wide, nosep]
\item $F$ is \textit{proper} if $\dom F\ne\emptyset$ and $-\infty_Y\notin F(X)$,
\item $F$ is {\it $K$-convex} if $\epi F$ is a convex subset of $X\times Y$,
\item $F$ is {\it $K$-epi closed} if $\epi F$ is a closed subset of $X\times Y$,
\item $F$ is {\it positively $K$-lower semi-continuous } or {\it $K$-lsc}  ({\it $K$-upper semi-continuous},  resp.)  if $y^*\circ F$ is lsc ($y^*\circ F$ is usc, resp.) for all $y^*\in K^+\setminus \{0_{Y^*}\}$.
\end{itemize}
The \textit{conjugate mapping} of $F$ is the   
 set-valued map $F^*\colon \mathcal{L}(X,Y) \rightrightarrows Y^\bullet$ defined by \cite[Definition 2.8]{DGLM16}
$$F^*(L):=\wsup\{L(x)-F(x):x\in X\}. $$
 The \textit{domain} of $F^*$ is
$\dom F^*:=\big\{L\in\mathcal{L}(X,Y):F^*(L)\ne \{+\infty_Y\}\big\}$, the \textit{$K$-epigraph}  of $F^*$ is 
\begin{align*}
\epi F^*&:=\big\{(L,y)\in\mathcal{L}(X,Y)\times Y:y\in F^*(L)+K\big\}. 
\end{align*}

It is useful to mention   that  if $F\colon X\to Y^\bullet$ is  a proper mapping then 
$\epi F^* $ is a closed subset of $\mathcal{L}  (X, Y) \times Y$  \cite[Lemma 3.5]{epi}. Moreover, it is established in \cite[Theorem 3.1]{DGLM16} that 
\begin{equation}\label{eq111}
	 (L,y)\in \epi F^* \quad\Longleftrightarrow\quad \Big(y-L(x)+F(x) \notin -\inte K,\;
		\forall x\in \dom F\Big). 
\end{equation}
The indicator mapping $I_{D}\colon X\rightarrow {Y}^{\bullet }$ of a set $D\subset X$
is defined by 
\begin{equation*}
I_{D}(x)=\left\{ 
\begin{array}{ll}
0_{Y}, & \text{if }x\in D, \\ 
+\infty _{Y}, & \text{otherwise.}%
\end{array}%
\right.
\end{equation*}%
In the case where $Y=\mathbb{R}$, $I_{D}$ collapses to  the usual indicator function $i_D$ of the set $D$ and the conjugate mapping $F^\ast$ collapses to the Fenchel conjugate function $F^\ast : X^\ast \to \R\cup\{ \pm \infty\}$ with 
$F^\ast (x^*) = \sup_{x \in X} [\la x^* , x \ra - F(x)] $ for all $ x^* \in X^\ast$. 
As usual, by $\Gamma (X)$ we denote the set of all proper, convex and lsc functions on $X$.  

 Let $S\not=\emptyset $ be a  convex cone in $Z$ and $\leqq _{S}$ be the usual ordering on 
$Z$\ induced by the cone $S$, i.e., 
$
z_{1}\leqq _{S}z_{2}\ \text{ if and only if } \ z_{2}-z_{1}\in S.$
We also enlarge $Z$ by attaching a greatest element $+\infty _{Z}$ and a
smallest element $-\infty _{Z}$  which do not
belong to $Z$, and define $Z^{\bullet }:=Z\cup \{-\infty _{Z},\ +\infty
_{Z}\}$. In $Z^{\bullet }$ we adopt the same  conventions as in \eqref{conv}.
Moreover, we recall the {\it cone of positive operators}  (see \cite{AliBurk}, \cite{DGLM16})  and the {\it cone of weak positive operators} \cite{epi} respectively,  as follows:
\begin{eqnarray*}
\mathcal{L}_{+}(S,K)&:=&\{T\in \mathcal{L}(Z,Y):\ T(S)\subset K\} \ \text{ and } \\
 \mathcal{L}^w_{+}(S,K) &:=& \{T\in\mathcal{L}(Z,Y):T(S)\cap (-\inte K)=\emptyset\}.
\end{eqnarray*}%
Lastly, for $T\in \mathcal{L}(Z,Y)$ and $G\colon X\to Z^\bullet$, the composite function $T\circ
G \colon X\rightarrow {Y}^{\bullet }$  is defined as follows: 
\begin{equation*}
(T\circ G)(x)=\left\{ 
\begin{array}{ll}
T(G(x)), & \text{if }G(x)\in Z, \\ 
+\infty _{Y}, & \text{if $G(x)=+\infty _{Z}.$}%
\end{array}%
\right.
\end{equation*}

\section{Basic tools: Sectional convexity and sectional closedness }
In this section, we will introduce {notions} of generalized  convexity {and closedness}, namely, the so-called ``sectional convexity" {and ``sectional closedness'', respectively},  and establish some basic properties of {these notions}.  
We introduce  the notions of $S^+$-uniformly usc of a mapping and uniformly $S^+$-concave of a family of mappings with their basic properties.  

\subsection{Uniform $S^+$-concavity  and $S^+$-uniformly  upper semi-continuity}

Let $Z$ be a lcHtvs with a pre-order defined by a non-empty, closed  and convex cone $S \subset Z$, and    $\U$ be a topological space.

\begin{definition}\label{def5.1}
Let $\mathcal{G}, \mathcal{G}_\nu\colon  \U\to Z\cup \{+\infty_Z\}$ for all $\nu\in I$. 
\begin{itemize}[wide, nosep]
\item[$\bullet$] We say that the collection $(\mathcal{G}_\nu)_{\nu\in I}$ is {\it uniformly $S^+$-concave} if
\begin{gather}
\forall z_1^*,z_2^*\in S^+,\; \forall u_1,u_2\in \U, \; \exists z^*\in S^+,\; \exists u\in \U \textrm{ such that }\notag\\
(z^*_1\circ \mathcal{G}_\nu)(u_1)+(z_2^*\circ \mathcal{G}_\nu)(u_2)\le (z^*\circ \mathcal{G}_\nu)(u),\;  {\forall \nu\in I}.  \label{eq_32abcde}
\end{gather}

\item[$\bullet$] We say that $\mathcal{G}$ is {\it $S^+$-uniformly usc} if,  for any net $(z^*_\alpha, u_\alpha,r_\alpha)_{\alpha\in D}\subset S^+\times \U\times \RR$ and $(z^*,u,r)\in S^+\times \U\times \RR$,
$$ \begin{cases}
(z^*_\alpha\circ \mathcal{G})( {u_\alpha}) \ge r_\alpha,\; \forall \alpha \in D\\
z^*_\alpha \overset{*}{\rightharpoonup} z^*,\; u_\alpha\to u,\; r_\alpha\to r
\end{cases}
\Longrightarrow\; (z^*\circ \mathcal{G})(u)\ge r.
$$
\end{itemize}
\end{definition}

The next example illustrates the meaning of the concept  of  ``{uniformly  $S^+$-concave}''  and it is used in the proof of Corollary  \ref{cor_7.4zz}  in Section 7.

\begin{example}\label{ex5.1}
Consider the case when $Z=\mathbb{R}$ and $S=\mathbb{R}_+$, and then, $Z^*=\mathbb{R}$ and $S^+=\mathbb{R}_+$. Assume that $\U$ is a convex subset of some topological  vector space and let $g_\nu\colon \U\to \mathbb{R}\cup\{+\infty\}$ be a concave function for each $\nu\in I$. Then  $(g_\nu)_{\nu\in I}$ is {uniformly $\mathbb{R}_+$-concave}. In deed, take $\lambda_1,\lambda_2\ge 0$ and $u_1,u_2\in \U$, we find $\lambda\ge 0$ and $u\in \U$ such that 
\begin{equation}\label{eq_33m}
\lambda_1g_\nu (u_1)+\lambda_2g_\nu(u_2)\le \lambda g_\nu(u), \quad {\nu\in I}.
\end{equation}
If $\lambda_1=\lambda_2=0$, we just take $\lambda=0$ and $u=u_1$. Assume that $\lambda_1>0$ or $\lambda_2>0$, or equivalently, $\lambda_0:=\lambda_1+\lambda_2>0$. {For all  $\nu\in I$}, as $g_\nu$ is concave, one has
\begin{align*}
\lambda_1g_\nu(u_1)+\lambda_2 g_\nu (u_2)=\lambda_0\left(\frac{\lambda_1}{\lambda_0}g_\nu(u_1)+\frac{\lambda_2}{\lambda_0}g_\nu (u_2)\right)\le \lambda_0g_\nu\left(\frac{\lambda_1}{\lambda_0}u_1+\frac{\lambda_2}{\lambda_0}u_2\right)
\end{align*} 
(note that $\frac{\lambda_1}{\lambda_0}+\frac{\lambda_2}{\lambda_0}=1$). So,  \eqref{eq_33m} follows by taking  $\lambda=\lambda_0$ and $u=\frac{\lambda_1}{\lambda_0}u_1+\frac{\lambda_2}{\lambda_0}u_2$.

\end{example}

\begin{lemma}\label{rem_5.1zzz}
\begin{itemize}[wide, nosep]
\item[$(i)$] If $\mathcal{G}$ is   {$S^+$-uniformly usc} then $\mathcal{G}$ is  positively $S$-usc.
\item[$(ii)$] For $Z=\RR$ and $S=\RR_+$, $\mathcal{G}$ is  {$\RR_+$-uniformly usc} if and only if it is usc.

\item[$(iii)$]  {Assume that $\U_i$ is a topological space and that $\mathcal{G}_i\colon \U_i\to \mathbb{R}\cup\{+\infty\}$ is a usc function, for all $i=1,\ldots m$. Then, the mapping $\mathcal{G}\colon \U:= \prod_{i=1}^m \U_i\to \mathbb{R}^m\cup\{+\infty_{\mathbb{R}^m}\}$ defined by $\mathcal{G}((u_i)_{i=1}^m)=(\mathcal{G}_i(u_i))_{i=1}^m$ is  $\mathbb{R}^m_+$-uniformly usc.}
 \end{itemize}
\end{lemma}

\begin{proof}
\begin{itemize}[wide]
\item[$(i)$] Assume that $\mathcal{G}$ is  {$S^+$-uniformly usc}.
Take $z^*\in S^+\setminus\{0_{Z^*}\}$. As  $\mathcal{G}$ is  {$S^+$-uniformly usc} then, it holds, for all net $( u_\alpha,r_\alpha)_{\alpha\in D}\subset  \U\times \RR$ and $(u,r)\in \U\times \RR$,
 $$ \begin{cases}
(z^*\circ \mathcal{G})( {u_\alpha}) \ge r_\alpha,\; \forall \alpha \in D\\
 u_\alpha\to u,\; r_\alpha\to r
\end{cases}
\Longrightarrow\; (z^*\circ \mathcal{G})(u)\ge r
$$
(let $z^*_\alpha=z^*$ for all $\alpha\in D$). This yields that the set $\{(u,r)\in \U\times \RR: (z^*\circ \mathcal{G})(u)\ge r\}$ is closed which also means that $z^*\circ \mathcal{G}$ is usc. {So, $\mathcal{G}$ is positively $S$-usc}.

\item[$(ii)$] 
If $\mathcal{G}$ is  {$\RR_+$-uniformly usc} then, according to $(i)$, $\mathcal{G}$ is positively $\RR_+$-usc yielding that $\mathcal{G}$ is usc.
Conversely, assume  that $\mathcal{G}$ is usc,  we will prove that  $\mathcal{G}$ is  {$\RR_+$-uniformly usc}. 
For this, take  $(\lambda_\alpha, u_\alpha,r_\alpha)_{\alpha\in D}\subset \RR_+\times \U\times \RR$ and $(\lambda,u,r)\in \RR_+\times \U\times \RR$ satisfying
\begin{gather}
\lambda_\alpha \mathcal{G}( {u_\alpha}) \ge r_\alpha,\; \forall \alpha \in D\label{eq_27zzzz}\\
{\textrm{and }} \lambda_\alpha\to \lambda,\;  u_\alpha\to u,\; r_\alpha\to r,
\end{gather}
we need to show that  $\lambda \mathcal{G}(u)\ge r$.

If $\lambda>0$  then, for all $\alpha$ large enough, $\lambda_\alpha>0$  and hence  \eqref{eq_27zzzz} yields $\mathcal{G}(u_\alpha)\ge \frac{r_\alpha}{\lambda_\alpha}$. Passing to the limit one gets $\mathcal{G}(u)\ge\frac{r}{\lambda}$ (as $\mathcal{G}$ is usc), or equivalently, $\lambda \mathcal{G}(u)\ge r$.

If  $\lambda=0$ then  as $\mathcal{G}$ is usc and $u_\alpha\to u$, one has $\mathcal{G}(u_\alpha) < \mathcal{G}(u)+1$  for $\alpha$ large enough. So, from  \eqref{eq_27zzzz}, one has,   for all $\alpha$ large enough, 
$\lambda_\alpha(\mathcal{G}(u)+1)\ge r_\alpha.  $
This leads to $r\le 0$ (as $\lambda_\alpha\to 0$ and $r_\alpha \to r$), and hence, one has  $\lambda \mathcal{G}(u) = 0 \ge r$.
%\end{itemize}
%In brief, it always {holds} $\lambda \mathcal{G}(u)\ge r$ which means that $\mathcal{G}$ is {$\RR_+$-uniformly usc}.

\item[$(iii)$]  {Take the net $(z^*_\alpha,u_\alpha, r_\alpha)_{\alpha\in D}\subset \mathbb{R}_+^{m}\times \U\times \mathbb{R}$ and $(z^*,u, r)\in  \mathbb{R}_+^{m}\times \U\times \mathbb{R}$ with $z^*_\alpha=(\lambda_i^\alpha)_{i=1}^m$, $u_\alpha=(u^\alpha_i)_{i=1}^m$, $z^*=(\lambda_i)_{i=1}^m$ and $u=(u_i)_{i=1}^m$, and assume that
\begin{gather}
\sum_{i=1}^m \lambda_i^\alpha \mathcal{G}_i(u_i^\alpha)\ge r_\alpha, \;\forall \alpha\in D\label{eq_3.5zba}\\
\lambda^\alpha_i\to\lambda_i,\; u_i^\alpha\to u_i, \; \forall i=1,\ldots , m;\quad r_\alpha\to r.
\end{gather}
We will prove that $\sum_{i=1}^m \lambda_i  \mathcal{G}_i(u_i)\ge r.$}

As $r_\alpha\to r$, there is $M>0$ such that $\lvert r_\alpha\rvert \le M$ for all $\alpha\in D$. For all $i=1,\ldots, m$, as $\lambda_i^\alpha\to\lambda_i$, $u_i^\alpha\to u_i$, $\mathcal{G}_i$ is usc, there exists $M_i$ such that $\lambda_i^\alpha \mathcal{G}_i(u_i^\alpha)\le M_i$ for all $\alpha\in D$. 

Now, take $i\in\{1,\ldots, m-1\}$. It follows from \eqref{eq_3.5zba} that
$$ M_i \ge \lambda_i^\alpha \mathcal{G}_i(u_i^\alpha)\ge r_\alpha -  \sum_{j\in\{1,\ldots m\}\setminus\{i\}} \lambda_j^\alpha \mathcal{G}_j(u_j^\alpha) \ge -M-  \sum_{j\in\{1,\ldots m\}\setminus\{i\}} M_j,\quad \forall \alpha\in D$$
So, we can suppose that $\gamma_i^\alpha:=\lambda_i^\alpha \mathcal{G}_i(u_i^\alpha)\to \gamma_i$. One then has
\begin{gather*}
 \lambda_i^\alpha \mathcal{G}_i(u_i^\alpha)\ge \gamma_i^\alpha, \;\forall \alpha\in D  \ {\rm and} \ \ 
\lambda^\alpha_i\to\lambda_i,\; u_i^\alpha\to u_i, \; \gamma_i^\alpha\to \gamma_i.
\end{gather*}
which, together with the fact that $\mathcal{G}_i$ is usc, 
 yields $\lambda_i\mathcal{G}_i(u_i)\ge \gamma_i$ (see the proof of $(ii)$).

{
As $r_\alpha\to r$ and $\gamma_i^\alpha\to \gamma_i$ for all $i=1,\ldots m-1$, we have $r_\alpha-\sum_{i=1}^{m-1}\gamma_i^\alpha\to r - \sum_{i=1}^{m-1}\gamma_i$. Moreover, according to \eqref{eq_3.5zba}, 
$$\lambda_{m}^\alpha\mathcal{G}_m(u_{m}^\alpha)\ge r_\alpha - \sum_{i=1}^{m-1}\gamma_i^\alpha,\; \forall \alpha\in D.$$
Consequently, as  $\lambda_ {m}^\alpha\to \lambda_{m}$, $u^\alpha_{m}\to u_{m}$, and as  $\mathcal{G}_{m }$ is usc, it holds  $\lambda_{m}\mathcal{G}_{m}(u_{m})\ge r-\sum_{i=1}^{m-1}\gamma_i$ (see again the proof of $(ii)$). So,  $\sum_{i=1}^{m} \lambda_i  \mathcal{G}_i(u_i)\ge \sum_{i=1}^{m-1}\gamma_i+r-\sum_{i=1}^{m-1}\gamma_i=r$ and we are done. }
\qedhere
 \end{itemize}
\end{proof}

\subsection{Sectional convexity and sectional closedness in topological vector spaces } 

Let  $E$  be  a {topological} vector space with  $E_0$ being  its {closed} subspace.  

\begin{definition}[Sectional convexity] 
We say that the subset $N\subset E$ is  \it{$E_0$-sectionally convex}   if $N\cap (E_0+v)$ is convex  for all $v\in E$.
\end{definition}

It is worth noting that if $N$ is a convex set then $N\cap (E_0+v)$ is convex for all $v\in E$, and hence, $N$ is also a $E_0$-sectionally convex set.  The converse, however, in general is not true,  for instance,  if  $E=\RR^2$ and $E_0=\{0\}\times \mathbb{R}$, then the set $N=\{(\alpha,\beta)\in\RR^2: \alpha^2\le \beta \le \alpha^2+1\}$ is  $E_0$-sectionally convex but it is obviously not convex.

It is easy to see that  the intersection of all  $E_0$-sectionally  convex subsets of $E$ containing $\emptyset \ne N \subset E$ is an $E_0$-sectionally convex subset of $E$ that contains $N$, which  is called   the {\it $E_0$-sectionally convex  hull} of $N$, and denoted by $\sco_{E_0} N$. Clearly, $\sco_{E_0} N$ is the smallest $E_0$-sectionally convex subset of  $E$ that contains $N$.  Moreover,  it is easy to see that  

$\bullet$ $\sco_{E_0} N\subset \co N$,  

$\bullet$ $N$ is $E_0$-sectionally  convex if and only if  $\sco_{E_0}N=N$, 

$\bullet$ If $N$ is $E_0$-sectionally convex then $\cl N$ is $E_0$-sectionally convex.  

 Moreover, when  $E_0=E$, the concepts  of ``$E$-sectionally convex'' and ``$E$-sectionally convex hull'' go back  to the usual  ones  ``convex'' and ``convex hull'' in convex analysis, respectively.

The next proposition gives  a presentation of $E_0$-sectionally convex hull of a set via the convex hull.

\begin{proposition}\label{pro_form_sco}
Let  $\emptyset \ne N\subset E$. Then  
\begin{equation}\label{eqn_pfs1}
\sco_{E_0} N = \bigcup_{v\in E}\co(N\cap (E_0+v)).
\end{equation}
\end{proposition}
\begin{proof} Denote the set in right-hand side of \eqref{eqn_pfs1} by $M$. To prove \eqref{eqn_pfs1}, it is sufficient to check that (i) $N\subset M$, (ii)  $M$ is a $E_0$-sectionally convex set, and (iii)  $M\subset M'$ for all $E_0$-sectionally convex subset $M' $ of $E$ that contains $N$.

\indent ${\rm (i)}$ \ As $\bigcup_{v\in E}(E_0+v)=E$, one has 
\begin{align*}
N&=N\cap \left(\bigcup_{v\in E}(E_0+v)\right)
=\bigcup_{v\in E}[N\cap (E_0+v)]
 \subset \bigcup_{v\in E}\co [N\cap (E_0+v)]=M.
\end{align*}

\indent ${\rm (ii)}$ \ 
We now prove that $M$ is $E_0$-sectionally convex.  For this, take any $\bar v\in E$, we will show that $M\cap (E_0+\bar v)$ is a convex set.
 Let us represent $M=M_1\cup M_2$ with
$$M_1 := \bigcup_{v\in E_0+\bar v}\co(N \cap (E_0+v)),\quad M_2:=\bigcup_{v \in E\setminus (E_0+\bar v) }\co (N\cap (E_0+v)).$$
As $E_0$ is a subspace of $E$, it is easy to check that $E_0+v=E_0+\bar v$ whenever $v\in E_0+\bar v$ and $(E_0+v)\cap (E_0+\bar v)=\emptyset$ whenever $v\notin E_0+\bar v$. This entails $M_1 = \co(N\cap (E_0+\bar v))$ and
\begin{align*}
M_2\cap (E_0+\bar v) &=\left(\bigcup_{v \in E\setminus (E_0+\bar v) }  \co(N\cap(E_0+v))\right)\cap (E_0+\bar v)\\
&\subset  \left(\bigcup_{v \in E\setminus {(E_0+\bar v)} }  (E_0+v)\right)\cap (E_0+\bar v)=\emptyset . 
\end{align*}
Note that the last inclusion follows from the fact that
\begin{equation}\label{eq_3.2az}
\co(N\cap (E_0+v))\subset E_0+v,\quad \forall v\in E
\end{equation}
as  $E_0+v$ is a convex subset containing $N\cap (E_0+v)$.
%\red{[is there  "co" in the big UNION?. If we use "co" then why we have "$= \emptyset$"] }
Consequently,
$$M\cap (E_0+\bar v)=\Big(\co (N\cap (E_0+\bar v)) \cap (E_0+\bar v)\Big)\cup \emptyset=\co (N\cap (E_0+\bar v)) $$
(see \eqref{eq_3.2az}).
So, $M\cap (E_0+\bar v)$ is a convex set.

\indent ${\rm (iii)}$ \  Now, assume that $M'$ is {an} $E_0$-sectionally convex subset containing $N$, we will show that $M\subset M'$. Take an arbitrary $w\in M$. Then, {by the definition of $M$}, there exists $v\in E$ such that $w\in\co(N\cap (E_0+v)).$ On the other hand, as $M'$ is {an} $E_0$-sectionally convex set containing $N$, the set $M'\cap (E_0+v)$ is convex and contains  $N\cap (E_0+v)$. So, $M'\cap (E_0+v)\supset \co(N\cap (E_0+v))$ which yields $w\in M'$. The proof is complete.
\end{proof}

\begin{definition}[Sectional closedness] 
We say that the subset $N\subset E$ is  \it{$E_0$-sectionally closed}   if for all $v\in E$ the set   $N\cap (E_0+v)$ is closed in $E$.
\end{definition}

It is also easy to see that a closed subset of $E$ is always  $E_0$-sectionally closed subset of  $E$  for any closed subspace $E_0$ of $E$.  In general, however, the converse is not true. For instance,  consider $E=\mathbb{R}^2$ and $E_0=\{0\}\times \mathbb{R}$, then the set $(0,1)\times [0,1]$ is $E_0$-sectionally closed but it is not a closed subset of $\R^2$.

For $N\subset E$,  the intersection of all $E_0$-sectionally closed (resp., $E_0$-sectionally closed and convex) subset of $E$ containing $N$ is called  the {\it $E_0$-sectional closure} (resp., {\it $E_0$-sectionally closed and convex hull of $N$}) of $N$, and is denoted by $\scl_{E_0} N$ (resp., by  $\sclco_{E_0} N$).

\begin{proposition} 
\label{pro_3.2zab}
Let $\emptyset \ne N\subset Y$.
\begin{itemize}[nosep]
\item[$(i)$] $\scl_{E_0} N = \bigcup_{v\in E}\cl(N\cap (E_0+v))$,  
\item[$(ii)$] If $N$ is $E_0$-sectionally convex then $\scl_{E_0} N$ is $E_0$-sectionally convex,
\item[$(iii)$] $ \sclco_{E_0} N  =     \scl_{E_0}(\sco_{E_0} N) =  \bigcup_{v\in E}\cl\co(N\cap (E_0+v))  $.
 \end{itemize} 
\end{proposition}

\begin{proof}%{\red{ \bf NOTE for Long:  change and check carefully this proof.}} 
$(i)$ The conclusion follows from  the same argument as in the proof of Proposition \ref{pro_form_sco}.

$(ii)$ Assume that $N$ is a $E_0$-sectionally convex.  
 Take $\bar v\in E$ and  we will prove that  $(\scl_{E_0} N) \cap (E_0+\bar v)$ is convex.
 It follows from the same argument as in the second  part   of the  proof  of Proposition \ref{pro_form_sco} and from (i) that
 \[
 \scl_{E_0} N  =  \Big[\bigcup_{v\in E_0+\bar v}\cl(N\cap (E_0+v)) \Big] \bigcup  \Big[\bigcup_{v\in E\setminus (E_0+\bar v)}\cl(N\cap (E_0+v)) \Big],
 \]
 and hence, 
\begin{align*}
(\scl_{E_0} N) \cap (E_0+\bar v)= \Big[ M_1\cap  (E_0+\bar v)\Big] \cup \Big[ M_2\cap  (E_0+\bar v)\Big],  
\end{align*}
where  $M_1:=\bigcup_{v\in E_0+\bar v}\cl(N\cap (E_0+v))$ and $M_2:=\bigcup_{v\in E\setminus (E_0+\bar v)}\cl(N\cap (E_0+v))$.
Use the similar argument as in the proof of  Proposition \ref{pro_form_sco} we can show that $M_1\cap (E_0+\bar v)=\cl(N\cap (E_0+\bar v))$ and $M_2\cap (E_0+\bar v)=\emptyset$, and hence, $(\scl_{E_0} N) \cap (E_0+\bar v)= \cl(N\cap (E_0+\bar v))$. On the other hand, as $N$ is $E_0$-sectionally convex, the set $N\cap (E_0+\bar v)$ is convex. So, $  \cl(N\cap (E_0+\bar v))$ is convex, as well, and we are done.

$(iii)$   Observe that   $\scl_{E_0}(\sco_{E_0} N)=\sclco_{E_0} N$. 
Indeed, as $\sclco_{E_0} N$ is a $E_0$-sectionally  convex subset containing $N$, one has $\sclco_{E_0} N\supset \sco_{E_0} N$. Note that $\sclco_{E_0} N $ is also $E_0$-sectionally  closed, so $\sclco_{E_0} N  \supset \scl_{E_0}(\sco_{E_0} N) $.

Conversely, by Proposition \ref{pro_3.2zab}$(ii)$, $\scl_{E_0} (\sco_{E_0} N)$ is a $E_0$-sectionally  closed and $E_0$-sectionally  convex subset containing $N$ yielding $\scl_{E_0}(\sco_{E_0} N) \supset \sclco_{E_0} N$.

Lastly, the equality $\sclco_{E_0} N = \bigcup_{v\in E}\cl\co(N\cap (E_0+v))  $  follows from the same argument as in the proof of Proposition \ref{pro_form_sco}. 
\end{proof}

\subsection{Sectional convexity and sectional closedness of epigraphs of conjugate mappings}

Let  $k\in Y\setminus\{0_Y\}$. For $x^*\in X^*$, we define  the mapping $k \cdot x^*\colon X\to Y$   by 
\begin{equation} \label{kx*}
(k\cdot x^*)(x)=\langle x^*,x\rangle k,\;  \forall  x\in X. 
\end{equation}  

Throughout this paper,  we are dealing with    the space  $E=\L(X,Y)\times Y$ and its subspace 
\begin{equation} \label{Ek}
E_k=  k\cdot (X^*\times \RR):=   \{(k\cdot x^*, rk): x^*\in X^*,\; r\in\mathbb{R}\}.   \end{equation} 
So, for the sake of simplicity, 
a subset $\mathcal{E}\subset \L(X,Y)\times Y$ is   $E_k$-sectionally  convex  then we say  that it is {\it $k$-sectionally convex}  and for the $E_k$-sectionally  convex hull of $\mathcal{E}$, we write  $\sco_k \mathcal{E}$ (instead of $\sco_{E_k}\mathcal{E}$) and call it $k$-sectionally  convex hull of $\mathcal{E}$. The same way applies  to   the    ``$E_k$-sectional closedness''  of $\mathcal{E}$ as well.

Turning back  to the  case when  $Y=\mathbb{R}$,  for each $\alpha\in \RR\setminus\{0\}$ (playing the role of $k$), one has $\alpha\cdot (X^*\times \RR)+(\bar x^*,\bar r)=X^*\times \mathbb{R} $ for all    $(\bar x^*,\bar r)\in X^*\times \mathbb{R}$. So, for given $\mathcal{E}\subset X^*\times \RR$, it holds $\mathcal{E}\cap (\alpha\cdot (X^*\times \RR)+(\bar x^*,\bar r))=\mathcal{E}$ for any $(\bar x^*,\bar r)\in X^*\times \mathbb{R}$. In  other words, in this case,  the notions of  ``$\alpha$-sectionally convex'', ``$\alpha$-sectionally closed'',  ``$\alpha$-sectionally convex  hull",  and ``$\alpha$-sectional closure"    collapse  to the  usual    ``convex", ``closed", ``convex hull",  and ``closure"   in convex analysis, respectively.

It is worth noticing that  $\epi F^*$, in general,  is not  a convex  subset of $\L(X, Y) \times Y$   even when $F$ is {a} linear continuous mapping  (see   \cite[Example 2.6]{epi}). However, as we will see in the next  proposition,  it is  always $k$-sectionally convex for any  $k\in K\setminus\{0_Y\}$.  

\begin{proposition}\label{epirco}
Let $F\colon X\to Y^\bullet$ be a proper mapping.  Then,   $\epi F^*$  is  $k$-sectionally  convex for each {$k\in   K\setminus\{0_Y\}$}.   
\end{proposition}

\begin{proof}  Let  $k\in  K\setminus\{0_Y\}$,  $(L, y) \in \mathcal{L}(X,Y)\times Y$, and  let  $E_k=k\cdot (X^*\times \RR)$ as in \eqref{Ek}. We will prove that $(\epi F^*)\cap [E_k+(L,y)]$ is a convex subset of $\L(X,Y)\times Y$.
For this, take  $a_1,a_2\in  {(\epi F^*)\cap [E_k+(L,y)]}$,   $\lambda\in ]0,1[$, it suffices to show that  $\lambda a_1+(1-\lambda)a_2 \in  \epi F^*$.    
As $a_i\in  E_k+(L,y)$, there exists $(x^*_i,r_i)\in  X^*\times \mathbb{R}$ such that
$a_i=k(x^*_i,r_i)+(L,y)$, $i = 1, 2$. On the other hand, 
as  $a_i\in \epi F^*$,  one has (see \eqref{eq111}) 
\begin{equation*}
y+r_ik -L(x)-\langle x_i^*, x\rangle k+F(x) \notin -\inte K,\; \forall x\in \dom F.\label{ptms15}
\end{equation*}
Then, according to Lemma \ref{pro_2.1aaa}, we get
 $$y+[\lambda r_1+(1-\lambda)r_2]k-L(x)-\langle \lambda x^*_1+(1-\lambda)x^*_2, x\rangle k+F(x)\notin -\inte K,\; \forall x\in \dom F$$
which, again by   \eqref{eq111},  yields  $\lambda a_1+(1-\lambda)a_2 \in  \epi F^*$ and 
the proof is complete.
\end{proof}

It is worth observing  that $k\cdot (X^*\times \RR)=(-k)\dot (X^*\times \RR)$ for any {$k\in K$}. So, by Proposition  \ref{epirco},  $\epi F^*$ is $k$-sectionally  convex for all {$k\in [K\cup(-K)]\setminus\{0_Y\}$}. However, the conclusion might not be true  when $k\notin K\cup (- K)$ as shown in  the next example.

\begin{example}%(\cite{Robust}  ?)
Take $X=\mathbb{R}$, $Y=\mathbb{R}^2$, $K=\mathbb{R}^2_+$, $F\colon \mathbb{R}\to \mathbb{R}^2$ the null mapping. Then $\mathcal{L}(X,Y)=\mathbb{R}^2$ and we get from Example 2.2 in \cite{epi} that $\epi F^* =\bigcup_{i=1}^4 N_i$ with
\begin{align*}
N_1 &= \{ (0,0,y_1,y_2):y_1\ge 0 \ \text{or} \  y_2\ge 0\}, \ \ 
N_2 = \{ (\alpha,\beta ,y_1,y_2):\alpha\beta<0, y_2\ge \frac{\beta}{\alpha} y_1\},\\
N_3 &= \{ (\alpha,0,y_1,y_2):\alpha\ne 0, y_2\ge 0\},\ \
N_4 = \{ (0,\beta,y_1,y_2):\beta\ne 0, {y_1\ge 0}\}.
\end{align*}
Now, take   $k=(1,-1)$, $E_k=k\cdot (X^*\times \RR)$, $L=(0,0)$ and $y=(0,-1)$. Then 
$E_k+(L,y)=\{(\alpha,- \alpha, y_1, -y_1-1): \alpha, y_1\in \mathbb{R}\}$, and  hence,
$$(\epi F^*)\cap [E_k+(L,y)] =  \{(0,0, y_1, -y_1-1): y_1\in \mathbb{R},\; y_1\ge 0\ {\rm or }\  y_1\le -1\}, $$
showing that   $(\epi F^*)\cap [E_k+(L,y)]$ is not a convex set, and consequently, $\epi F^*$ is not $(1,-1)$-sectionally convex.
\end{example}

\begin{proposition}\label{epirco2}
	Let $F\colon X\to Y^\bullet, \; G\colon X\to Z^\bullet$ be proper mappings, and $k\in  K\setminus\{0_Y\}$.  
	Then the set $\bigcup_{z\in S^+} \epi (F+(k\cdot z^*)\circ G)^*$ is  $k$-sectionally  convex, where   $k \cdot z^* \in \L(Z,Y) $ is the mapping defined  as in \eqref{kx*}. 
\end{proposition}

\begin{proof}   Let   $E_k=k\cdot (X^*\times \RR)$ (defined by \eqref{Ek}), and set 
$$\mathcal{M}_k:=\bigcup_{z^*\in S^+} \epi (F+(k\cdot z^*)\circ G)^*.$$ 
Take  $(L, y) \in \mathcal{L}(X,Y)\times Y$ and we  will show  that the set ${\mathcal{M}_k}\cap [E_k+(L,y)]$ is  convex.
The proof goes parallelly as that of Proposition \ref{epirco}.  Take  $a_1,a_2\in {\mathcal{M}_k}\cap (E_k+(L,y))$,   $\lambda\in ]0,1[$, it suffices to show that  $\lambda a_1+(1-\lambda)a_2 \in   {\mathcal{M}_k}$.    
As $a_i\in {\mathcal{M}_k}\cap (E_k+(L,y))$,  there exist $(x^*_i,r_i)\in  X^*\times \mathbb{R}$ and $z^*_i\in S^+$ such that
$a_i=k(x^*_i,r_i)+(L,y)   \in \epi (F + (k\cdot z^*)\circ G)^\ast $ for $i = 1, 2$.  By  \eqref{eq111}, one has 
\begin{align*}
y+r_ik -L(x)-\la x_i^*,x\ra k+F(x)+  (z^*_i\circ G)(x)k \notin -\inte K,\; \forall x\in X, \forall i = 1, 2.\label{ptms15bis}
\end{align*}
It now follows from Lemma  \ref{pro_2.1aaa} that 
\begin{eqnarray}\label{eq_10bis}
\!\!\! \lambda a_1\!+\! (1-\lambda )a_2\! = \! 
y\!+\! \bar r k -L(x)\!-\!\la \bar x^*, x\ra k\! + \! F(x)\!  + \! (\bar z^*\circ G)(x)k\notin \! -\inte K,\! \forall x\in X, 
\end{eqnarray}
where $ \bar r =\lambda r_1+(1-\lambda)r_2$, $\bar x^*= \lambda x_1^* + (1-\lambda_2) x_2^*$, and   $\bar z^*:=\lambda z^*_1+(1-\lambda )z^*_2$ (note that $\bar z^*  \in S^+$  
  as $z^*_1,z^*_2\in S^+$).   Again, by  \eqref{eq111},    \eqref{eq_10bis} means that  $\lambda a_1+(1-\lambda )a_2\in    \epi(F+(k\cdot \bar z^*)\circ G)^*\subset \mathcal{M}_k$  and  the proof is complete.
\end{proof}

\begin{remark}\label{rem_3.1_nw}
According to \cite[Lemma 3.6]{epi},  for  any proper mapping  $F\colon X\to Y^\bullet$, 
the set  $ \epi F^*$ is always closed and consequently, it is $k$-sectionally  closed for each  $k\in Y\setminus\{0\}$.
\end{remark}

\medskip

\section{ Epigraphs of conjugate mappings via sectionally convex hulls}

We are now concerning  the \textit{robust vector optimization problem} of the model \cite{Robust}, \cite{DL17VJM}:   
\begin{equation}
\begin{array}{rrl}
\mathrm{(RVP)}\ \  & \wmin\left\{ F(x): x\in C,\; G_u(x)\in
-S,\;\forall u\in\mathcal{U}\right\} , & 
\end{array}
\label{rvop}
\end{equation}
where, as in previous sections,    $X, Y, Z$ are lcHtvs, $K$ is a  closed and convex cone in $Y$ with nonempty interior,  
and $S$ is a closed, convex cone in $Z$,
$\mathcal{U}$ is an \textit{uncertainty set}, $F\colon X\rightarrow {Y}^\bullet,$ $G_u\colon X\rightarrow Z^\bullet$ are proper mappings,  and  $\emptyset \ne C\subset X$.
The feasible set of (RVP) is 
\begin{equation} \label{feas-A}
A:=C\cap \left(\bigcap_{u\in\mathcal{U}}G_u^{-1}(-S)\right).\end{equation} 
We assume through out this paper that 
$A\cap\dom F\ne \emptyset$.

In this section we will establish various representations of the epigraph of the conjugate mapping $(F +I_A)^\ast$,  $\epi (F + I_A)^\ast$.  The representations hold under ``closure" signs and without any  constraint qualification conditions and so they are  called asymptotic representations.   These representations  will play a crucial role in establishing  the main results of the next sections:  robust vector Farkas-type results and duality for the problem (RVP).  

Concerning the problem(RVP), we recall  the  \emph{qualifying set} \cite{Robust}  and the  \emph{weak qualifying set} \cite{DL17VJM} defined  respectively  as follows:  
\begin{align}
\mathcal{A} &:=\bigcup_{(T,u)\in \mathcal{L}_{+}(S,K)\times \U}\epi (F+I_C+T\circ G_u)^*,    \label{A}\\
\mathcal{B}&:=\bigcup_{(T,u)\in \mathcal{L}^w_{+}(S,K)\times \U}  
   \left( \bigcap_{v\in I^*_{-S}(T)} [\epi(F+I_C+T\circ G_u )^*+(0_{\mathcal{L}},v) ]\right).\label{B}
\end{align}
For  $k\in \inte K$, we now introduce  another  {\it  qualifying set}  $ \A_{k} $  defined by  
\begin{equation} \label{Ak} 
\A_{k}:= \bigcup_{(z^*,u)\in S^+\times \U}\epi (F+I_C+(k\cdot z^*)\circ G_u)^*,
\end{equation}
where  $k\cdot z^*\colon Z\to Y$ defined by $(k\cdot z^*)(z)=\la z^*,z\ra k$ for all $z \in Z$ (see also \eqref{kx*}). 

In the case when $Y=\mathbb{R}$ and $K=\mathbb{R}_+$  all the sets $\A$,  $\B$, and $\A_k$ collapse  to the usual qualifying set (see \cite{DMVV17-Robust-SIAM})  $\bigcup_{(z^*,u)\in S^+\times \U} \epi (F+i_C+z^*\circ G_u)^*.$

The   relations between these  sets and $\epi(F+I_A)^*$ are given in the next    proposition.  
\begin{proposition} 
	\label{epi_include}
	It holds
		$\epi (F+I_A)^*\supset \B\supset \A \, \supset\A_k$   for all $k\in \inte K$.
\end{proposition}

\begin{proof} It is easy to see that $k\cdot z^*\in \L_+(S,K)$ whenever $z^*\in S^+$ and $k\in \inte K$. So,  
\begin{equation} \label{eqaa} 
\A_k\subset \A , \ \forall   k\in \inte K.
\end{equation} 
Now, for each $u\in \U$, let $A_u:= C\cap G_u^{-1}(-S)$.  Repeat the same argument as in  the first part of the proof of Theorem 4.2 (namely, the proof of (27))  in \cite{epi}\footnote{Note that in the first part  of the proof of 	\cite[Theorem 4.2]{epi}  no assumptions  on the convexity or closedness of the mappings $F$ and $G$ are needed. }, we get for each $u \in \U$, 
	\begin{align*}
		\epi(F+I_{A_u})^*&\supset \bigcup_{T\in \mathcal{L}^w_{+}(S,K)}  
		{\left(\bigcap_{v\in I^*_{-S}(T)}[\epi(F+I_C+T\circ G_u )^*+(0_{\mathcal{L}},v) ]\right)}\\
		&\supset \bigcup_{T\in \mathcal{L}_{+}(S,K)}\epi (F+I_C+T\circ G_u)^*.
	\end{align*}
Consequently,  
\begin{equation} \label{eqbb} 
\bigcup_{u\in \U}\epi (F+I_{A_u})^*\supset \B \supset \A. \end{equation} 
 On the other hand,
	as   $A=\bigcap_{u\in \U} A_u$,  according to  \eqref{eq111}, one has 
	\begin{eqnarray*} 
	(L, y) \in  \bigcup_{u\in \U}  \epi(F+I_{A_u})^\ast \ &\Longrightarrow \ &   \exists u \in \U:  	(L,y) \in  \epi(F+I_{A_u})^*   \nonumber \\
	 \ &\Longrightarrow \ &   \exists u \in \U:  y - L(x) + F(x)  \notin - \inte K,\; \forall x \in A_u \cap \dom F   \nonumber \\
	 	 \ &\Longrightarrow \ &    y - L(x) + F(x)  \notin - \inte K,\; \forall x \in A \cap \dom F    \nonumber \\	
	 \ &\Longrightarrow \ & (L,y) \in 	\epi(F+I_A)^\ast, 
	\end{eqnarray*} 
	which means that 
	\begin{equation} \label{eqcc} 
		\bigcup_{u\in \U}  \epi(F+I_{A_u})^\ast  \subset \epi(F+I_A)^\ast . \end{equation} 
		The conclusion now follows from \eqref{eqaa}, \eqref{eqbb}, and \eqref{eqcc}. 
	\end{proof}

\begin{lemma}
	\label{lem_scalar-representing}
	Assume that $f\in \Gamma(X)$ and  that $C$ is a nonempty closed convex subset of $X$. Assume further that   $G_u$ is proper,  $S$-convex and 
	$S$-epi closed for each $u\in \mathcal{U}$, and that $A\cap\dom f\ne\emptyset$ (where $A$ is given by \eqref{feas-A}). Then\footnote{ This would be an elementary result in the study of robust optimization problems. However, to the  surprise of the authors, we could not find it in the references we had in hand and so we insert a short proof here.} 
		\begin{align}
		\epi (f+i_A)^*&=\cl\co \left(\bigcup_{(z^*,u)\in S^+\times \mathcal{U}}
			\epi (f+i_C+z^*\circ G_u)^*\right).	\label{eq_13a}
	\end{align}
\end{lemma}

\begin{proof}
	For each $u\in \U$, set 
	$A_u=C\cap G_u^{-1}(-S)$. Then 
 $A=\bigcap_{u\in\U}A_u$ and 
	$\sup_{u\in\U}(f+i_{A_u})(x)=(f+i_A)(x)$ for all $x\in X$.
	As  for each $u \in \U$, $G_u$ is proper,  $S$-convex and 
	$S$-epi closed  and  $A\cap\dom f\ne\emptyset$,  one has $f+i_{A_u}\in \Gamma (X)$ for all $u\in \U$. 
	Now, take $x_0\in A\cap \dom f$ (note that  $A\cap \dom f\ne\emptyset$), one gets 
	$\sup_{u\in\U}(f+i_{A_u})(x_0)=(f+i_A)(x_0) <+\infty$.  So, according to \cite[Lemma 2.2]{GLi}, it holds
	
	\begin{equation}  
		\label{prop41a}
		\epi (f+i_A)^*=\cl \co \left[ \bigcup_{u\in \U} \epi (f+i_{A_u})^*\right].
	\end{equation}
On the other hand,  for each $u\in \U$, it follows from \cite[Theorem 8.2]{Bot10} that 
	\begin{equation}  
		\label{prop41w}
		\epi (f+i_{A_u})^*=\cl \left[ \bigcup_{z^*\in S^+} \epi (f+i_C+z^*\circ G_u)^*\right].
	\end{equation}
	The equality \eqref{eq_13a} now follows from combining \eqref{prop41a} to \eqref{prop41w}.
\end{proof}

 We are now in a position to prove the main results of this section. Our purpose is to generalize the representation in Lemma \ref{lem_scalar-representing} to the vector {case}.    The difficulty  in such a generalization   is that   the set $\epi(F+I_A)^*$ in general is  not  convex  \cite[Example 2.6]{epi}, and hence, {it is almost no hope for a representation of the  same form  as in  \eqref{eq_13a}.}  { Fortunately, with the help of } Proposition \ref{epirco},    \eqref{eq_13a}   can be generalized {with  the use of the  $k$-sectionally convex hull,}   as  shown in the next theorem. 

 We  need   a hypothesis on the convexity of  data from (RVP) first.

\medskip 
\begin{tabular}{l|l}
$(\mathcal{H}_0)$  &$F$ is $K$-convex and positively $K$-lsc, $G_u\colon X\to Z$ is $S$-convex and \\ &$S$-epi-closed for all $u\in\mathcal{U}$, and $C$ is nonempty, closed and convex.
\end{tabular}

\begin{theorem}\label{thm_epi_presenting1}
Assume that  $(\mathcal{H}_0)$ holds. Then, for each $k\in\inte K$, one has 
$$ \epi (F+I_A)^\ast= \cl(\sco\nolimits_k{\A_k}).$$
\end{theorem}

\begin{proof}
 Take $\bar k\in \inte K$. By  Proposition \ref{epi_include}, 
 $ \epi(F+I_A)^*\supset \A_{\bar k}$.  Moreover,   by  \cite[Lemma 3.6]{epi} and   Proposition \ref{epirco},  the set  $ \epi(F+I_A)^*$ is  closed and  $\bar k$-sectionally convex, respectively, and so,        
 \begin{equation}\label{eqepiaa} 
 \epi(F+I_A)^\ast  \supset \cl (\sco_{\bar k} {\A_{\bar k}}). 
 \end{equation}
  We now prove  the converse of the  inclusion \eqref{eqepiaa}.   For this, take  $(\bar L,\bar y)\in\epi (F+I_A)^*$ and we will prove that $(\bar L,\bar y)\in \cl(\sco_{\bar k} {\A_{\bar k}}).$    Let us structure the rest of our proof in five  steps.   
  
 $\bullet$ \emph{Step 1.}  {\it Prove that the set $(\bar L-F)(A\cap\dom F)- \inte K$ is convex.} Observe firstly that   as  $G_u $ is $S$-convex for all $u\in \mathcal{U}$,  and $C$ is convex,  the feasible set     $A$ is convex. 
Also, $F-\bar L$ is a $K$-convex mapping. Thus,    $(F-\bar L)(A\cap\dom F)+ \inte K$ is convex (see \cite[Remark 4.1]{epi}), and so is    $(\bar L-F)(A\cap\dom F)- \inte K$.   

$\bullet$ \emph{Step 2.}
As   $(\bar L,\bar y)\in \epi(F+I_A)^*$, it follows from characterizing \eqref{eq111} that 
 \begin{equation}
	\label{eqa} 
	\bar y\notin (\bar L-F)(A\cap\dom F)-\inte K. 
\end{equation} 
So, applying the convex separation theorem  \cite[Lemma 3.4]{Rudin91}, there  {is}  $y^*\in Y^*\setminus\{0_{Y^*}\}$ such that
\begin{equation}\label{1}
y^*(w)< y^*(\bar y),\quad \forall w\in (\bar L-F)(A\cap\dom F)- \inte K.
\end{equation}
It then follows from \cite[Lemma 3.3]{epi}  that  
\begin{gather}
y^*\circ (\bar L-F)(x)\le y^*(\bar y),\; \forall x\in A\cap \dom F,\label{1.1}\\
y^*\in K^+\; \textrm{ and }\; y^*(k^\prime )>0,\; \forall k^\prime \in\inte K.\label{1.2}
\end{gather}

 $\bullet$ \emph{Step 3.} It is easy to see that
 \eqref{1.1} is equivalent to $y^*(\bar y)\ge (y^*\circ F+i_A)^*(y^*\circ \bar L)$, or equivalently, $(y^*\circ \bar L, y^*(\bar y))\in \epi\nolimits  (y^*\circ  F+i_A)^*$.
On the other hand, as  $y^*\in K^+\setminus\{0_{Y^*}\}$ and $F$ is $K$-convex and positively $K$-lsc, one has {$y^*\circ F\in \Gamma (X)$} and now Lemma \ref{lem_scalar-representing},  applying to $f=y^*\circ F$, gives us 
\begin{equation}
\label{eqthm41}
  \epi (y^*\circ F+i_A)^*=\cl\co  \tilde \A , 
\end{equation}
where $\tilde \A :=  \bigcup_{(z^*,u)\in S^+\times \mathcal{U}}\epi (y^*\circ F+i_C+z^*\circ G_u)^*.$  Since  $(y^*\circ \bar L, y^*(\bar y))\in \epi\nolimits  (y^*\circ  ~F~+~i_A)^*$, it follows from \eqref{eqthm41} that 
   there exist a net $(x^*_\alpha,r_\alpha)_{\alpha\in D}\subset \co\tilde \A$ \footnote{{For the sake of simplicity,  we write $(x^*_\alpha,r_\alpha)_{\alpha\in D}$ {for } $((x^*_\alpha, r_\alpha))_{\alpha\in D}$.}} such that $(x^*_\alpha,r_\alpha)\to (y^*\circ \bar L, y^*(\bar y))$.  So, for each  $\alpha\in D$,  there are a finite index set $I_\alpha$, and finite sequences $(z^*_{\alpha_i})_{i\in I_\alpha} \subset S^+$, $(u_{\alpha_i})_{i\in I_\alpha}\subset \mathcal{U}$,
 $(x^*_{\alpha_i})_{i\in I_\alpha}\subset X^*$, $(r_{\alpha_i})_{i\in I_\alpha}\subset \mathbb{R}$ and
$(\lambda_{\alpha_i})_{i\in I_\alpha}\subset \mathbb{R}_+\setminus\{0\}$ such that
$\sum_{i\in I_\alpha}\lambda_{\alpha_i}=1$, $\sum_{i\in I_\alpha} \lambda_{\alpha_i}(x^*_{\alpha_i}, r_{\alpha_i})=(x^*_\alpha,r_\alpha)$, and 
\begin{gather}
(x^*_{\alpha_i},r_{\alpha_i})\in \epi\nolimits(y^*\circ F+i_C + z^*_{\alpha_i}\circ G_{u_{\alpha_i}}),\; \forall i\in I_\alpha.  \label{e10}
\end{gather}

	 $\bullet$  \emph{Step 4.} As $\bar k\in \inte K$, it follows from   \eqref{1.2} that  $y^*(\bar k)>0$.  For each $\alpha\in D$ and $i\in I_\alpha$, let us  define the elements $y_{\alpha_i}\in Y$,  $\tilde z^*_{\alpha_i} \in Z^\ast$,   and 
	the mapping $L_{\alpha_i}\colon X\rightarrow Y$, respectively by  
	\begin{equation}
		\label{defmaps}
		y_{\alpha_i}:=\bar y+\frac{r_{\alpha_i}\!\!-\!y^*(\bar y)}{y^*(\bar k)}\bar k,\
	\tilde z^*_{\alpha_i}(z):= \frac{z^*_{\alpha_i}(z)}{y^*(\bar k)}, \ 	  L_{\alpha_i}(x):=\bar L(x)+\frac{x_{\alpha_i}^*(x)\!-\!y^*\circ \bar L(x)}{y^*(\bar k)} \bar k.	\end{equation}
	Then, it is easy to check that 
	\begin{equation} \label{thm41eqaa}
	 L_{\alpha_i}\in \mathcal{L}(X,Y),\;  \tilde z^*_{\alpha_i}\in S^+,\;
%$T_{\alpha_i} \in \mathcal{L}_+(S,K)$,
	y^*(y_{\alpha_i})=r_{\alpha_i},\; y^*\circ L_{\alpha_i}=x^*_{\alpha_i},\; 
	y^*\circ (\bar k\cdot \tilde z^*_{\alpha_i})= z^*_{\alpha_i}, 
	\end{equation} 
	and  
	\begin{equation}
		\label{limits}
		\Big(\sum_{i\in I_\alpha}\lambda_{\alpha_i}L_{\alpha_i},\sum_{i\in I_\alpha}
			\lambda_{\alpha_i} y_{\alpha_i}\Big) \longrightarrow {(\bar L,\bar y)}.
	\end{equation}

	We now show that  for  each $\alpha\in D$ and $i\in I_\alpha$,  $(L_{\alpha_i},y_{\alpha_i})\in \A_{\bar k}$.   It follows from \eqref{thm41eqaa} and  \eqref {e10} that
	$$ y^*(y_{\alpha_i})\ge (y^*\circ F+i_C + y^*_{\alpha_i}\circ   { (\bar k\cdot \tilde z^*_{\alpha_i})}
		\circ G_{u_{\alpha_i}})^*(y^*\circ L_{\alpha_i}),$$
	which is equivalent to
	\begin{equation*}
		y^*(y_{\alpha_i})\ge y^*\circ L_{\alpha_i}(x)-y^*\circ F(x) -y^*\circ  { (\bar k\cdot \tilde z^*_{\alpha_i})}
			\circ G_{u_{\alpha_i}}(x),\quad \forall x\in C\cap \dom F,
	\end{equation*}
	or equivalently,
	\begin{equation*}
		y^*\left[L_{\alpha_i}(x)- F(x)- { (\bar k\cdot \tilde z^*_{\alpha_i})}\circ G_{u_{\alpha_i}}(x) -y_{\alpha_i} \right]\le 0, \quad \forall x\in C\cap \dom F.
	\end{equation*}
	The last inequality, together with \eqref{1.2}, yields
	$$y_{\alpha_i}\notin L_{\alpha_i}(x)-F(x)- { (\bar k\cdot \tilde z^*_{\alpha_i})}\circ G_{u_{\alpha_i}}(x) - \inte K,
		\quad \forall x\in C\cap \dom F,$$
	which in turn   yields {(see \eqref{eq111})} 
	\begin{equation} 
		\label{eqnew}
		(L_{\alpha_i},y_{\alpha_i})\in {\epi} (F+I_{C} +  { (\bar k\cdot \tilde z^*_{\alpha_i})}\circ G_{u_{\alpha_i}})^*\subset {\A_{\bar k}}. 
	\end{equation}

	$\bullet$ \emph{Step 5.} Let $E_{\bar k}=\bar k\cdot (X^*\times \RR)$. 
	According to Proposition \ref{pro_form_sco}, one has 
	{$$\sco_{\bar k} \A_{\bar k}:=\bigcup_{(L,y)\in\L(X,Y)\times Y} \co [\A_{\bar k}\cap (E_{\bar k}+(L,y))].$$} 
	For all $\alpha\in D$,  it follows from \eqref{eqnew} and \eqref{defmaps} that
	$$(L_{\alpha_i},y_{\alpha_i})\in {\A_{\bar k}}\cap (E_{\bar k}+(\bar L,\bar y)),\quad \forall i\in I_\alpha,$$
	and hence, if take $L_\alpha:= \sum_{i\in I_\alpha}\lambda_{\alpha_i}L_{\alpha_i}$ and 
	$y_\alpha:=\sum_{i\in I_\alpha}\lambda_{\alpha_i} y_{\alpha_i}$ then it holds 
 	$(L_\alpha,y_\alpha)\in \co [{\A_{\bar k}}\cap (E_{\bar k}+(\bar L,\bar y))]\subset \sco_{\bar k} {\A_{\bar k}}$. 
	From \eqref{limits}, $(\bar L,\bar y)  = \lim_{\alpha \in D}  (L_\alpha,y_\alpha)$, showing that 
	 $(\bar L,\bar y )\in \cl(\sco_{\bar k} {\A_{\bar k}})$ and we are done.  
\end{proof}

%\red{
%\begin{remark}
%It follows from Theorem \ref{thm_epi_presenting1} that if $(\mathcal{H}_0)$ then $\cl(\sco\nolimits_k\A)$ do not depend on taking the member $k\in \inte K$, or in other word,
%$$ \cl(\sco\nolimits_k\A)=\cl(\sco\nolimits_{k'}\A),\quad \forall k,k'\in\inte K.$$
%\end{remark}
%}

\begin{corollary}
\label{rem_4.1zba}
Assume that  $(\mathcal{H}_0)$ holds. Then one has  
\begin{equation} \label{rem410}
\epi(F+I_A)^*=\scl_{k}(\sco_{k} \A_k),\; \forall k\in \inte K.
\end{equation}
 \end{corollary}
 \begin{proof}   
Indeed, take $\bar k\in \inte K$, according to Theorem \ref{thm_epi_presenting1}, 
\begin{equation}\label{rem41a} 
 \epi(F+I_A)^*=\cl(\sco_{\bar k}\A)\supset \scl_{\bar k} (\sco_{\bar k} \A)
 \end{equation}
(note that all closed subsets are $k$-sectionally closed). 

On the other hand, according to Proposition \ref{pro_3.2zab}$(iii)$, one has 
	$$\scl_{\bar k}(\sco_{\bar k} \A_{\bar k})=\bigcup_{(L,y)\in\L(X,Y)\times Y} \cl\co [\A_{\bar k}\cap (E_{\bar k}+(L,y))].$$
For each $(\bar L,\bar y)$, we can see from the proof of Theorem \ref{thm_epi_presenting1} that $(\bar L,\bar y)$  is the limits of the net $(L_\alpha,y_\alpha)_{\alpha\in D}$ with $(L_\alpha,y_\alpha)\in \co [\A_{\bar k}\cap (E_{\bar k}+(\bar L,\bar y))] $ for all $\alpha\in D$. This yields $(\bar L,\bar y)\subset \cl\co[\A_{\bar k}\cap (E_{\bar k}+(\bar L,\bar y))]\subset \scl_{\bar k}(\sco_{\bar k} \A_{\bar k})$. So, 
\begin{equation} \label{rem41b} 
\epi(F+I_A)^*\subset \scl_{\bar k}(\sco_{\bar k} \A_{\bar k}), 
\end{equation} 
and hence \eqref{rem410} follows from \eqref{rem41a} and \eqref{rem41b}. 
\end{proof}

\begin{theorem}\label{thm_epi-presenting2}
Assume that $(\mathcal{H}_0)$ holds. 
Then, for all $k\in\inte K$, 
\begin{equation}\label{eq_26bis}
 \epi (F+I_A)^*= \cl(\sco\nolimits_k\B)={\cl(\sco\nolimits_k\A)}.
\end{equation}
\end{theorem}

\begin{proof}
Take  $\bar k\in \inte K$.  It follows from Proposition \ref{epi_include} that
$$\epi (F+I_A)^*\supset \B\supset\\A \supset \A_{\bar k} $$  
and as  $\epi (F+I_A)^*$ is $\bar k$-sectionally convex   (by  Proposition \ref{epirco})  {and closed}
(see  \cite[Lemma 3.6]{epi}), one gets 
$$\epi (F+I_A)^*\supset   \cl(\sco\nolimits_k\B)     \supset\   \cl(\sco\nolimits_k\A)  \supset  \cl(\sco\nolimits_k   \A_{\bar k}) .$$
On the other hand, under the assumption $(\mathcal{H}_0)$,  Theorem \ref{thm_epi_presenting1}  gives that   $\epi(F+I_A)^*=\cl(\sco_{\bar k} \A_{\bar k})$,  and hence, \eqref{eq_26bis} follows.
\end{proof}

In the case with the absence of  the  uncertainty, i.e., the  uncertainty set $\U$ is a singleton, Theorems \ref{thm_epi_presenting1} -  \ref{thm_epi-presenting2} collapse to the ones that 
cover  both  Theorems 4.1 and 4.2 in   \cite{epi}.

\begin{corollary}
Assume that $F\colon X\to Y^\bullet$ is a proper $K$-convex and positively $K$-lsc mapping, that $G\colon X\to Z$ is a proper $S$-convex and $S$-epi closed mapping, and that $C$ is nonempty, closed and convex. Assume further that $B\cap \dom F\ne \emptyset$ where $B:=C\cap G^{-1}(-S)$.
Then 
\begin{align*}
\epi (F+I_B)^*
&=\cl \left(\bigcup_{T\in \mathcal{L}_{+}(S,K) }\epi (F+I_C+T\circ G)^*\right)\notag\\
&=\cl\left(\bigcup_{T\in \mathcal{L}^w_{+}(S,K)}\bigcap_{v\in I^*_{-S}(T)} [\epi(F+I_C+T\circ G )^*+(0_{\mathcal{L}},v) ]\right) \label{eq_30zzzzz} \\
&= \cl \left( \bigcup_{z^*\in S^+ }\epi (F+I_C+(\bar k\cdot z^*)\circ G)^* \right). \notag  
\end{align*}
\end{corollary}

\begin{proof} Let $\bar k \in \inte K$. 
In the case where the uncertainty set $\U$ is a singleton, the qualifying sets  $\A, \A_{\bar k}, \B $ become the following sets, respectively 
\begin{align*}
\mathcal{\tilde \A}&=\bigcup_{T\in \mathcal{L}_{+}(S,K) }\epi (F+I_C+T\circ G)^*,\\
{\mathcal{\tilde \A}_{\bar k}}&   ={\bigcup_{z^*\in S^+ }\epi (F+I_C+(\bar k\cdot z^*)\circ G)^*,}\\
\mathcal{\tilde \B}&=\bigcup_{T\in \mathcal{L}^w_{+}(S,K)}\bigcap_{v\in I^*_{-S}(T)} [\epi(F+I_C+T\circ G )^*+(0_{\mathcal{L}},v) ].
\end{align*}
In such a case ($\U$ is a singleton)   Proposition \ref{epirco}  gives ${\tilde \A }_{\bar k}\subset {\tilde \A }\subset {\tilde \B}\subset\epi (F+I_B)^*$, which, together with the fact that 
 $\epi (F+I_B)^\ast$ is closed (see \cite[Lemma 3.6]{epi}), leads to  
\begin{equation}\label{eq_4.20az}
\cl {\tilde \A}_{\bar k}\subset\cl {\tilde \A}\subset \cl {\tilde \B}\subset \epi (F+I_B)^*.
\end{equation}
On the other hand,  according to Proposition \ref{epirco2}, the set ${\tilde\A}_{\bar k}$ is $\bar k$-sectionally convex, and hence,  $\sco_{\bar k}{\tilde \A_{\bar k} } =  {\tilde \A_{\bar k}}$. 
Now,  Theorem \ref{thm_epi_presenting1} yields  $\epi (F+I_B)^*=\cl(\sco_{\bar k} \tilde \A_{\bar k}) = \cl \tilde \A_{\bar k}$, which, together with \eqref {eq_4.20az} 
leads to 
\begin{equation*}
\cl {\tilde \A}_{\bar k}\subset\cl {\tilde \A}\subset \cl {\tilde \B}\subset \epi (F+I_B)^* \subset  \cl {\tilde \A}_{\bar k}, 
\end{equation*}
and the conclusion follows. 
\end{proof}

\section{ Robust vector Farkas-type results }\label{sec3}

We retain all the notations used in the previous sections  and consider   the robust vector optimization problem (RVP) defined by   \eqref{rvop}   
with its feasible set $A$ as in \eqref{feas-A} and the assumption   
$A\cap\dom F\ne \emptyset$.    Consider the qualifying sets $\A$, $\B$ and $\A_{k}$ (for some $k \in \inte K$) defined respectively  by  \eqref{A}, \eqref{B}, and \eqref{Ak}. Moreover, 
we say that $\A$ ($\B$, $\A_{k}$, respectively)  is $k$-sectionally 
			convex  and  closed regarding $\V\times \W$ if $\cl (\sco_k \A)\cap (\V\times \W)=\A\cap (\V\times \W)$  ($\cl (\sco_k \B)\cap (\V\times \W)=\B\cap (\V\times \W)$, $\cl (\sco_k \A_{k} )\cap (\V\times \W)=\A_k\cap (\V\times \W)$, respectively). Let $  \emptyset \ne  \V\subset {\mathcal{L}(X,Y)}$ and $  \emptyset \ne \W\subset Y$.

We now establish some principles and results on  $(\V,\W)$-stable Farkas lemma  for vector-valued systems concerning the robust vector optimization problem $({\rm RVP})$.  In the first one, Theorem \ref{thm_CSFI}, for the sake of completeness,   we quote  $[({\rm a})\Leftrightarrow ({\rm b})]$ from   \cite[Theorem 3.2(ii)]{Robust}.    
Note also that  Theorem \ref{thm_CSFI} extends     \cite[Theorems 1,2]{DL17VJM}.

\begin{theorem}[Principles of stable robust vector Farkas lemma I] 
	\label{thm_CSFI} 
	Consider the following {statements}
	\begin{itemize}[nosep]
		\item[$\rm(a)$]  $\epi(F+I_A)^{\ast }\cap (\V\times \W)=\A\cap (\V\times \W)$,

	\item[$\rm(b)$]  $\epi(F+I_A)^{\ast }\cap (\V\times \W)=\B\cap (\V\times \W)$,

		\item[$\rm(c)$]    For any   $(L,y)\in \V\times \W$, the next two assertions are equivalent: 
		\begin{itemize}[nosep] 
			\item[$(\alpha)$] $G_u(x)\in -S,\;  x\in C, \; \forall u\in \mathcal{U}\; 
				\Longrightarrow \; y-L(x)+F(x)\notin -\inte K,$
			
			\item[$(\beta)$] There exist $u\in\U$ and $T\in \L_+(S,K)$ such that 
				$$F(x)+T\circ G_u(x)-L(x)+y\notin -\inte K,\;\forall x\in C.$$
		\end{itemize}

		\item[$\rm(d)$]   For any   $(L,y)\in \V\times \W$, the next two assertions are equivalent: 
		\begin{itemize}[nosep] 
			\item[$(\alpha)$] $G_u(x)\in -S,\;  x\in C, \; \forall u\in \mathcal{U}\; 
				\Longrightarrow \; y-L(x)+F(x)\notin -\inte K,$
				\item[$(\gamma)$] There exist $u\in\U$ and $T\in \L_+^w(S,K)$ such that 
					$$F(x)+T\circ G_u(x)-L(x)+y\notin -T(S)-\inte K,\;\forall x\in C.$$
			\end{itemize}
	\end{itemize}
Then, $[({\rm a})\Leftrightarrow ({\rm c})]$ and  $[({\rm b})\Leftrightarrow ({\rm d})]$.
\end{theorem}

\begin{proof}  The first equivalence,  $[({\rm a})\Leftrightarrow ({\rm c})]$,  is  Theorem 3.2(ii) in \cite{Robust}. However, for the sake of completeness, we 
give briefly the proof here. It is easy to  see that 
 $(\alpha)$ is equivalent to  $(L,y) \in \epi (F +I_A)^\ast $ while $(\beta)$ is equivalent to $(L,y) \in \A$. So,  $[({\rm a})\Leftrightarrow ({\rm c})]$ holds.  The proof of the second one,   $[({\rm b})\Leftrightarrow ({\rm d})]$,   can be obtained by using a similar way using the weak cone of positive operators  $\L_+^w(S,K)$ instead of $\L_+(S,K)$.     \end{proof}

For each  $k\in \inte K$, {recall that} (see \eqref{Ak}) 
$$\A_{k}= \bigcup_{(z^*,u)\in S^+\times \U}\epi (F+I_C+(k\cdot z^*)\circ G_u)^*.$$
%\red{with} $k\cdot z^*\colon Z\to Y$ defined by $(k\cdot z^*)(z)=\la z^*,z\ra k$. 

Another principle for stable robust vector Farkas lemma based on $\A_k$ is given in the next theorem.

\begin{theorem}[Principle of stable robust vector Farkas lemma II] 
	\label{thm_CSFIbbis} 
	Let $k\in \inte K$, the following {statements} are equivelent
	\begin{itemize}[nosep]
		\item[$({\rm a}_k)$]  $\epi(F+I_A)^{\ast }\cap (\V\times \W)=\A_k\cap (\V\times \W)$,

		\item[$({\rm c}_k)$]      For any   $(L,y)\in \V\times \W$, the next two assertions are equivalent: 
		\begin{itemize}[nosep] 
			\item[$(\alpha)$] $G_u(x)\in -S,\;  x\in C, \; \forall u\in \mathcal{U}\; 
				\Longrightarrow \; y-L(x)+F(x)\notin -\inte K,$
			
			\item[$(\delta)$] There exist $u\in\U$ and $z^*\in  S^+$ such that 
				$$F(x)+(z^*\circ G_u)(x)k-L(x)+y\notin -\inte K,\;\forall x\in C.$$
		\end{itemize}
	\end{itemize}
\end{theorem}

\begin{proof}
The proof is similar to that  of Theorem \ref{thm_CSFI}. It is clear that $(\alpha)$ is equivalent to $(L,y) \in \epi (F + I_A)^\ast$ (by \eqref{eq111})  and 
$(\delta)$ is equivalent to $(L,y) \in \A_{k}$. So, $({\rm a}_k)$ and $({\rm c}_k)$ are equivalent. 
\end{proof}

Now, { we are  ready to establish  } principles of $(\V,\W)$-stable robust vector Farkas lemma  in convex setting (i.e., under  the hypothesis $(\mathcal{H}_0)$). These results are obtained by combining Theorem \ref{thm_CSFI} and  the results on representations  of $\epi(F+I_A)^*$ provided in Section 4. 

\begin{theorem}[Principles of stable robust convex  vector Farkas lemma I] 
	\label{thm_FLI}
Assume that   that the hypothesis $(\mathcal{H}_0)$   holds.  Consider the following {statements}:
 	\begin{itemize}[nosep]
		\item[$({\rm a}_1)$] $ \exists k\in \inte K$ s.t.   $\A$ is $k$-sectionally 
			convex  and  closed regarding $\V\times \W$,

		\item[$({\rm b}_1)$] $ \exists k\in \inte K$ s.t.  $\B$ is $k$-sectionally 
			convex  and  closed regarding $\V\times \W$.

	\end{itemize}
Then, $[({\rm a}_1)\Leftrightarrow ({\rm c})]$ and  $[({\rm b}_1)\Leftrightarrow ({\rm d})]$, where {$(\rm c)$ and} $(\rm d)$ are the statements  in  Theorem \ref{thm_CSFI}.
\end{theorem}

\begin{proof} %\begin{itemize}[wide]
%\item	{\it Proof of $[({\rm a}_1)\Leftrightarrow ({\rm c})]$:}
Let $k\in  \inte K$. As $(\mathcal{H}_0)$ holds, it follows from Theorem 
	\ref{thm_epi-presenting2} that 
	$$ \epi (F+I_A)^*= \cl(\sco\nolimits_k\A)=    \cl(\sco\nolimits_k\B). $$
	So, the statements $({\rm a}_1)$ and $({\rm b}_1)$ are respectively equivalent to
		$$\epi(F+I_A)^{\ast }\cap (\V\times \W)=\A\cap (\V\times \W)\ \ 
{\rm  and}\ \ 
\epi(F+I_A)^{\ast }\cap (\V\times \W)=\B\cap (\V\times \W)$$
for any $ \emptyset \ne \V \subset  \L(X, Y)$ and any $\emptyset \ne \W \subset Y$. 
	The conclusion now follows from Theorem \ref{thm_CSFI}. 
%\item 	{\it Proof of $[({\rm b}_1)\Leftrightarrow ({\rm d})]$:}
	%Similar to the proof of Theorem \ref{thm_FLI} $[({\rm a}_1)\Leftrightarrow ({\rm c})]$, use Theorem  \ref{thm_epi-presenting2} instead of Theorem  \ref{thm_epi_presenting1}.\qedhere 
%\end{itemize}
\end{proof}

For each  $k\in \inte K$, {recall that}  the set $\A_k$ is defined by  \eqref{Ak}.  
 We are now seeking for other alternative qualifying conditions    based on the set $\A_{k}$ that guarantee the previous versions of robust vector Farkas lemmas.

\begin{theorem}[Stable robust convex vector  Farkas lemma I]
\label{thm_5.3}
Assume    that the hypothesis $(\mathcal{H}_0)$   and the following condition   $({\rm a}_2)$  hold:
\begin{itemize}[nosep]
\item[$({\rm a}_2)$]  $ \exists k\in \inte K$ s.t.     {$\A_{k}$} is $k$-sectionally convex and closed regarding $\V\times \W$.
\end{itemize}
Then, the assertions $({\rm c})$, $({\rm d})$ in Theorem \ref{thm_CSFI} hold.
\end{theorem}

\begin{proof}

%For $k\in \inte K$, it holds $k\cdot z^*\in \L_+(S,K)$ whenever $z^*\in S^+$.
 %So, $\red{\A_{k}}\subset \A$ and then, 

Firstly, according to  Proposition \ref{epi_include}, one has
\begin{equation}\label{eq_34zzz}
\A_{k}\cap (\V\times \W)\subset \A\cap (\V\times \W)\subset \B  \cap (\V\times \W)    \subset \epi (F+I_A)^*\cap (\V\times \W).
\end{equation}
On the other hand, under the hypothesis $(\mathcal{H}_0)$, Theorem \ref{thm_epi_presenting1}  yields  $\epi (F+I_A)^*= {\cl(\sco_{k} \A_{k})}$, which  
combining with 
 $({\rm a}_2)$, one gets   
$$   \epi (F+I_A)^* \cap (\V\times \W) =     \cl(\sco_{k_0}  {\A_{k_0}})\cap (\V\times \W)={\A_{k_0}} \cap (\V\times \W),$$
This, together with \eqref{eq_34zzz}, assures  that $({\rm a})$ {and  $({\rm b})$} in Theorem \ref{thm_CSFI} hold. The conclusion now follows from Theorem \ref{thm_CSFI}. 
\end{proof}

\begin{theorem}[Principles for stable robust convex vector Farkas-lemma II] 
	\label{thm_FLIzzz}  Let $k\in \inte K$ and     
assume that    the hypothesis $(\mathcal{H}_0)$   holds.  Then the following {statements} $({\rm a}^\prime _k)$ and $({\rm c}_k)$   are equivalent: 
\begin{itemize}[nosep]
\item[$({\rm a}^\prime_k)$] $\A_k$ is $k$-sectionally 
			convex  and  closed regarding $\V\times \W$, 
\item[$({\rm c}_k)$]    For any   $(L,y)\in \V\times \W$, the next two assertions are equivalent: 
		\begin{itemize}[nosep] 
			\item[$(\alpha)$] $G_u(x)\in -S,\;  x\in C, \; \forall u\in \mathcal{U}\; 
				\Longrightarrow \; y-L(x)+F(x)\notin -\inte K,$
			
			\item[$(\delta)$] There exist $u\in\U$ and $z^*\in  S^+$ such that 
				$$F(x)+(z^*\circ G_u)(x)k-L(x)+y\notin -\inte K,\;\forall x\in C.$$
		\end{itemize}
\end{itemize}

\end{theorem}

\begin{proof}  As   $(\mathcal{H}_0)$   holds,   it follows from Theorem 4.1 that $\epi (F+I_A)^*=\cl(\sco_k \A_k)$,  and then $({\rm a}^\prime_k)$ is nothing else but $({\rm a}_k)$ in Theorem 5.2. The conclusion now follows from Theorem 5.2. 
\end{proof}

Some sufficient conditions for  $k$-sectional convexity and $k$-sectional closedness of the sets $\A_k$ (with $ k \in \inte K$), $\A$, and $\B$  will be given below. We first consider  some more  assumptions:

\begin{tabular}{l|l}
$(\mathcal{H}_1)$&  The collection $\Big(u\mapsto G_u(x)\Big)_{x\in C\cap\dom F}$ is  {uniformly $S^+$-concave},\\
\end{tabular}

\medskip 

\begin{tabular}{l|l}
$(\mathcal{H}_2)$&  The mapping $u\mapsto G_u(x)$ is {$S^+$-uniformly usc} for each  {$x\in C\cap\dom F$}.\\
\end{tabular}

\begin{proposition}\label{pro_5.1}
If $(\mathcal{H}_0)$ and  $(\mathcal{H}_1)$ hold then {$\A_k$} is $k$-sectionally convex for each  $k\in \inte K$.
\end{proposition}

\begin{proof}
 Take $k_0\in \inte K$ and $(L_0, y_0) \in \mathcal{L}(X,Y)\times Y$, we will prove that ${\A_{k_0}}\cap [{E_{k_0}}+(L_0,y_0)]$ is a convex set with ${E_{k_0}}=k_0\cdot (X^*\times \RR)$. For this, take  {$a_1,a_2\in \A_{k_0}\cap [E_{k_0}+(L_0,y_0)]$} and   $\lambda\in ]0,1[$, we show that  $\lambda a_1+(1-\lambda)a_2 \in  { \A_{k_0}}$.

For $i=1,2$, as $a_i\in  {\A_{k_0}}\cap [{E_{k_0}}+(L_0,y_0)]$, there exist $(x^*_i,r_i)\in  X^*\times \mathbb{R}$ and $(z^*_i,u_i)\in S^+\times \U$ such that
$a_i={k_0}(x^*_i,r_i)+(L_0,y_0)$ and (according to \eqref{eq111}) 
\begin{align*}
y_0+r_ik_0 -L_0(x)-\langle x_i^*, x\rangle k_0+F(x)+(z^*_i\circ G_{u_i})(x)k_0 &\notin -\inte K,\; \forall {x\in C\cap\dom F}.
\end{align*}
Lemma \ref{pro_2.1aaa} applying to $y=y_0-L_0(x)+F(x)$, $\alpha=r_1-\langle x^*_1, x\rangle+(z^*_1\circ G_{u_1})(x)$ and $\beta=r_2-\langle x^*_2, x\rangle+(z^*_2\circ G_{u_2})(x)$
yields
\begin{gather}
y_0+[\lambda r_1+(1-\lambda )r_2]k_0 -L_0(x)-\langle \lambda x_1^*+(1-\lambda )x_2^*, x\rangle k_0+F(x)\notag\\
+[\lambda (z^*_1\circ G_{u_1})(x)+(1-\lambda)(z^*_2\circ G_{u_2})(x)]k_0 \notin -\inte K,\quad \forall x\in C\cap \dom F.\label{eq_34aaa}
\end{gather}
On the other hand,  as $(\mathcal{H}_1)$ holds, there exists $(\bar z^*, \bar u)\in S^+\times \U$ such that
$$\lambda (z^*_1\circ G_{u_1})(x)+(1-\lambda)(z^*_2G_{u_2})(x)\le (\bar z^*\circ G_{\bar u})(x), \forall x\in C\cap \dom F,$$
or, equivalently, 
\begin{equation}
\Big((\bar z^*\circ G_{\bar u})(x)-[\lambda (z^*_1\circ G_{u_1})(x)+(1-\lambda)(z^*_2G_{u_2})(x)]\Big)k_0\in K.\label{eq_35aaa}
\end{equation}
By \eqref{eq_1aaa}, we get from  \eqref{eq_34aaa} and  \eqref{eq_35aaa},  
$$y_0+[\lambda r_1+(1-\lambda )r_2]k_0 -L_0(x)-\langle \lambda x_1^*+(1-\lambda )x_2^*, x\rangle k_0+F(x)+ (\bar z^*\circ G_{\bar u})(x){k_0}\notin -\inte K, $$
{for all $x\in C\cap\dom F$},  which means  that (see \eqref{eq111})  $\lambda a_1+(1-\lambda )a_2 =(L_0,y_0)+k_0[\lambda (x^*_1, r_1)+(1-\lambda)(x^*_2,r_2)]
\in \epi (F+(k_0\cdot \bar z^*)\circ G)^*\subset {\A_{k_0}}$ . The proof is complete.
\end{proof}

\begin{proposition}\label{pro_5.1zzzz}
Assume that  $\U$ is a compact space, that $Z$ is {a normed space}, and that $(\mathcal{H}_2)$ and the following Slater-type condition hold:
\begin{itemize}[nosep]
\item[$(C_0)$] $\forall u\in \U,\;\exists  x_u\in C\cap\dom F: G_u(x_u) \in - \inte S.$
\end{itemize}
Then, {$\A_{k}$} is $k$-sectionally closed for each  $k\in \inte K$
\end{proposition}

\begin{proof}
Take arbitrarily $k_0\in \inte K$ and $(L_0,y_0)\in \L(X,Y)\times Y$, we will prove that  $\A_{k_0}\cap (E_{k_0}+(L_0,y_0))$ is closed,  where $E_{k_0}=k_0\cdot (X^*\times \RR)$ (defined in  \eqref{Ek}).
For this, take $(L_\alpha, y_\alpha)_{\alpha\in D}\subset \A_{k_0}\cap (E_{k_0}+(L_0,y_0))$ such that $(L_\alpha, y_\alpha)\to (L,y)$, we need to show that $(L,y)\in \A_{k_0}\cap (E_{k_0}+(L_0,y_0))$.

%Let structure the proof in 3 steps as follows:

%$\bullet$ 
%Firstly, as $0_Y\notin \inte K$, according to convex separation theorem  \cite[Lemma 3.4]{Rudin91}, there exists $y_0^*\in Y^*\setminus\{0_{Y^*}\}$ such that
%\begin{gather}\label{eq_36aaa}
%y_0^*(k)>0,\; \forall k \in\inte K,
%\end{gather} 
%and consequently, $y_0^*\in K^+$.

$\bullet$ Firstly, for all $\alpha\in D$,  as $(L_\alpha, y_\alpha)\in \A_{k_0}\cap (E_{k_0}+(L_0,y_0))$, there exist $(z^*_\alpha,u_\alpha)\in S^+\times \U$ and $(x^*_\alpha,r_\alpha)\in X^*\times \mathbb{R}$ such that
$(L_\alpha, y_\alpha)\in \epi (F+I_C+(k_0\cdot z^*_\alpha)\circ G_{u_\alpha})^*$ and $(L_\alpha,y_\alpha)=k_0(x^*_\alpha, r_\alpha)+(L_0,y_0)$. Then,  by  \eqref{eq111}, it holds
\begin{equation*}
y_0+r_\alpha k_0 -L_0(x) -\langle x^*_\alpha, x\rangle k_0 + F(x)+ (z^*_\alpha\circ G_{u_\alpha})(x) k_0\notin  -\inte K,\; \forall x\in C\cap\dom F,
\end{equation*}
or equivalently, 
 \begin{equation}\label{eq_36aab}
y_0  -L_0(x) \!   +\! F(x) \!    +\!  \Big(\!  r_\alpha  -\langle x^*_\alpha, x\rangle  + (z^*_\alpha\circ G_{u_\alpha})(x)\Big) \! k_0\notin  -\inte K, \forall x\in C\cap\dom F,
\end{equation}

$\bullet$ Next, for each $x\in C\cap\dom F$, according to Lemma \ref{lem_2.2zba}, there is $\alpha_x\in \mathbb{R}$ such that 
\begin{equation}
\label{eq_36aaa}
\alpha<\alpha_x \; \Longleftrightarrow\; y_0-L_0(x)+F(x)+\alpha k_0\in-\inte K.
\end{equation}
It follows from \eqref{eq_36aaa} and    \eqref{eq_36aab} that 
\begin{equation}\label{eq_37aaa}
r_\alpha-\langle x^*_\alpha, x\rangle+( z^*_\alpha\circ G_{u_\alpha})(x)  \ge \alpha_x, \  \forall x \in C\cap\dom F.  
\end{equation}

%For $z\in Z$ and $r>0$, denote by $B(z,r)$ the open ball of center $z$ and radius $r$.

$\bullet$ {\it We now prove that the net $(\lVert z^*_\alpha\rVert)_{\alpha\in D}$ is bounded, where $\lVert z^*_\alpha\rVert:=\sup_{\lVert z\rVert \le 1}\langle z^*_\alpha,z\rangle$.} 
Let assume by contradiction that $\lVert z^*_\alpha\rVert\to +\infty$. Without loss of generality  we can assume that $\lVert z^*_\alpha\rVert>0$ for all $\alpha\in D$, and hence, according to \eqref{eq_37aaa}, 
\begin{equation}\label{eq_37aaabb}
 \left(\tilde z^*_\alpha\circ G_{u_\alpha}\right)(x)  \ge \frac{1}{\lVert z^*_\alpha\rVert} (\alpha_x+ \langle x^*_\alpha,x \rangle-r_\alpha),\; \forall x\in C\cap\dom F, \;\forall \alpha\in D,
\end{equation}
 where $\tilde z^*_\alpha=\frac{1}{\lVert z^*_\alpha\rVert}z^*_\alpha$.

Note that $(L_\alpha, y_\alpha)  =k_0(x^*_\alpha, r_\alpha)+(L_0,y_0)   \to (L,y)$. We now prove that  there is $(\bar x^*,\bar r)\in X^*\times \mathbb{R}$ such that
\begin{equation}\label{eq_5.8zba}
 ( x^*_\alpha, r_\alpha)\to (\bar x^*,\bar r) \textrm{ and }(L,y)=k_0(\bar x^*, \bar r)+(L_0,y_0).
\end{equation}
 Indeed, take $\bar y^*\in Y^*$ such that $\bar y^*(k_0)=1$ (it is possible as $k_0\ne 0_Y$). As $(L_\alpha, y_\alpha)  =k_0(x^*_\alpha, r_\alpha)+(L_0,y_0)   \to (L,y)$, one has $k_0 (x^*_\alpha, r_\alpha) \to (L - L_0, y - y_0)$ or equivalently, $k_0\cdot x_\alpha^* \to L - L_0$ and $r_\alpha k_0  \to y - y_0$. 
 Apply $\bar y^* $ to these expressions, one gets 
\begin{eqnarray*} 
&&  \bar y^* \circ (k_0\cdot x_\alpha^*) = (\bar y^* (k_0)) x^*_\alpha = x^*_\alpha  \overset{*}{\rightharpoonup} \bar y^* \circ (L - L_0) =: \bar x^* \\ 
&&  \bar y^* \circ (k_0 r_\alpha) = (\bar y^*(k_0) r_\alpha = r_\alpha  \to \bar y^*(y - y_0) = : \bar r , 
\end{eqnarray*}  
 which gives $(L_\alpha, y_\alpha)  =k_0(x^*_\alpha, r_\alpha)+(L_0,y_0)   \to k_0 (\bar x^*, \bar r) + (L_0, y_0)$ and \eqref{eq_5.8zba} follows by the uniqueness of the limit.

On the other hand, as $\lVert \tilde z^*_\alpha\rVert=1$ for all $\alpha\in D$ by  Banach-Alaoglu theorem, without loss of generality,  we can assume that  $\tilde z^*_\alpha \overset{*}{\rightharpoonup}  \tilde z^*\in Z^*$  and as $\U$ is compact we also can assume (without loss of generality) that $u_\alpha\to \bar u\in \U$.
So, pass to the limit (with $\alpha \in D$) in \eqref{eq_37aaabb},  taking into account that $\lVert z^*_\alpha\rVert\to +\infty$ and  that $u\mapsto G_u$ is $S^+$-uniformly usc for all $x\in C\cap\dom F$ (by $(\mathcal{H}_2)$), one gets
\begin{equation}\label{eq_5.16az}
(\tilde z^*\circ G_{\bar u})(x)\ge 0,\; \forall x\in C\cap \dom F.
\end{equation}

Next, as $\tilde z^*_\alpha\in S^+$, $\lVert \tilde z^*_\alpha\rVert=1$ for all $\alpha\in D$, and $\tilde z^*_\alpha\overset{*}{\rightharpoonup} \tilde z^*$, it holds $\tilde z^*\in S^+\setminus\{0_{Z^*}\}$. Consequently, $\tilde z^*(s)>0$ for all $s\in\inte S$ which, together with \eqref{eq_5.16az}, yields
$$G_{\bar u}(x)\notin -\inte S,\; \forall x\in C\cap \dom F.$$
This contradicts $(C_0)$, and hence,  the net $(\lVert z^*_\alpha\rVert)_{\alpha\in D}$ is bounded.

$\bullet$ As $(\lVert z^*_\alpha\rVert)_{\alpha\in D}$ is bounded, we can assume that $z^*_\alpha\overset{*}{\rightharpoonup} \bar z^*\in S^+$. For each $x\in C\cap\dom F$, pass to the limit in \eqref{eq_37aaa},  with the noting that $u_\alpha\to \bar u\in \U$, $r_\alpha\to \bar r$, $x^*_\alpha\overset{*}{\rightharpoonup} \bar x^*$, and that $u\mapsto G_u$ is $S^+$-uniformly usc, one gets
\begin{equation}\label{eq_37aaabbb}
\bar r-\langle \bar x^*, x\rangle+( \bar z^*\circ G_{\bar u})(x)  \ge \alpha_x.
\end{equation}
This, together with \eqref{eq_36aaa}, accounts for
$$y_0    -L_0(x) + F(x)   + \Big[  \bar r  -\langle \bar x^*, x\rangle +(\bar z^*\circ G_{\bar u})(x)\Big] k_0\notin -\inte K, $$

So, $(L,y)=k_0(\bar x^*, \bar r)+(L_0,y_0)\in \epi (F+I_C+(\bar z^*\circ G_{\bar u}))^*\subset \A_{k_0}$, and hence, $(L,y)\in \A_{k_0}\cap (E_{k_0}+(L_0,y_0))$. The proof is complete.
\end{proof}

\begin{proposition}\label{cor_5.1zzz}
Assume that $\U$ is a compact space, that $Z$ is {a normed space}, and that the  hypotheses $(\mathcal{H}_0)$, $(\mathcal{H}_1)$,   $(\mathcal{H}_2)$,  and the Slater-type condition $(C_0)$ hold.
Then the sets $\A_k$ (for any $k\in \inte K$), $\A$, $\B$ are  $k$-sectionally convex and closed.  
\end{proposition}

\begin{proof}
Take $k_0\in \inte K$. 
According to Corollary  \ref{rem_4.1zba},
$\epi(F+I_A)^*=\scl_{k_0}(\sco_{k_0} \A_{k_0}).$
Moreover, it follows from Propositions \ref{pro_5.1}, \ref{pro_5.1zzzz} that $\A_{k_0}$ is $k_0$-sectionally convex and $k_0$-sectionally closed, and hence, $\epi(F+I_A)^*=\A_{k_0}$.  
On the other hand, according to Proposition   \ref{epi_include},
$\epi(F+I_A)^*\supset \B\subset \A\supset \A_{k_0}$. So, $\epi (F+I_A)^*=\B=\A=\A_{k_0}$.
The conclusion now follows from this and the fact that $\epi (F+I_A)^*$ is closed and $k$-sectionally convex.
\end{proof}

Combining Theorems  \ref{thm_CSFI}-\ref{thm_FLIzzz}  and Propositions \ref{pro_5.1}-\ref{cor_5.1zzz}, we get the following version of stable  robust vector Farkas lemma  in    
convex setting. 

\begin{theorem}[Stable robust convex vector  Farkas lemma II]
\label{pro_5.3zba}
Assume that $\U$ is a compact space and $Z$ is {a normed space}. Assume further   that the  hypotheses $(\mathcal{H}_0)$, $(\mathcal{H}_1)$, $(\mathcal{H}_2)$,  and the Slater-type condition $(C_0)$ hold.

Then,  for all $\V\times\W\subset\L(X,Y)\times Y$,  the three versions of  $\V\times\W$-stable robust Farkas lemma  described in  $({\rm c})$ and  $({\rm d})$ in Theorem \ref{thm_CSFI}, and $({\rm c}_k)$ (for arbitrary $k\in \inte K$) in Theorem \ref{thm_FLIzzz}  hold.
\end{theorem}

\begin{proof}
Take $\V\times \W\in \L(X,Y)\times Y$ and $k\in \inte K$. It follows from Proposition  \ref{cor_5.1zzz} that the qualifying conditions $(\rm a_1)$ and $(\rm b_1)$  in Theorem \ref{thm_FLI}, and $({\rm c}_k)$ in Theorem \ref{thm_FLIzzz}  hold. The conclusion follows from Theorems \ref{thm_FLI}, \ref{thm_FLIzzz}.
\end{proof}

\section{Duality for robust convex vector optimization problems}

  In this section, concerning    the robust vector optimization problem (RVP) defined by   \eqref{rvop}   
with its feasible set $A$ as in \eqref{feas-A} with  the assumption  that  
$A\cap\dom F\ne \emptyset$,{we define a new kind of Lagrange dual problems $({\rm RVD}^k)$ based on some results $k$-sectional convexity with  $ k \in \inte K$}, along with the Lagrangian robust dual problem  $({\rm RVD})$  
and the weak Lagrangian robust dual problem  $({\rm RVD}_w)$ introduced in \cite{DL17VJM}.   
We   will establish several robust  strong stable  duality results for the pairs  (RVP)-$({\rm RVD})$,  (RVP)-$({\rm RVD}_w)$, and (RVP)-$({\rm RVD}^k)$. The results on robust  strong stable  duality  for the pair   (RVP)-$({\rm RVD}^k)$ are new while the ones for other dual pairs are established under  qualification conditions which are different from \cite{DL17VJM}  and  some how are easier to check than the ones in \cite{DL17VJM}.

Recall that   $\A$, $\B$ and $\A_{k}$ (for some $k \in \inte K$) the qualifying sets defined respectively  by  \eqref{A}, \eqref{B}, and \eqref{Ak}.

\subsection{Lagrange duality for robust vector optimization problems}

 We consider the {\it Lagrangian robust dual problem} $({\rm RVD})$ and the {\it weak Lagrangian robust dual problem}     $({\rm RVD}_w)$  \cite{DL17VJM} of $\mathrm{(RVP)}$ defined   respectively by 
\begin{align*}
	&{({\rm RVD})}\quad\  \mathop{\wsup}\limits_{(T, u)\in    \L_+(S,K)\times \U} 
		\mathop{\winf}_{x\in C} (F+ T\circ G_u)(x)   ,\\
	&({\rm RVD}_w) \quad \mathop{\wsup}\limits_{(T,u)\in    \L^w_+(S,K)\times \U}  
		\mathop{\winf}_{(x,s)\in C\times S} \big[(F+ T\circ G_u)(x) + T(s)\big].
\end{align*}
 We say that {\it the robust strong   duality holds for the pair ${({\rm RVP})-({\rm RVD})}$} (resp., {\it for the pair ${({\rm RVP})-({\rm RVD}w)}$})  if  the sets of values of the two problems $({\rm RVP})$ and $({\rm RVD})$ (resp.,  of the two problems $({\rm RVP})$ and $({\rm RVD}_w)$) are equal together, that is, 
 \begin{equation} \label{eq_5.1}
	\wmax ({\rm RVD})  = \winf {\rm (RVP)}\qquad (\textrm{resp., } \wmax ({\rm RVD}_w)  = \winf {\rm (RVP)}) . 
\end{equation}

For $L\in \L(X,Y)$, we denote by  $({\rm RVP}^L)$ the  perturbed vector problem
 \begin{align}
\label{rvp_L}
({\rm RVP}^L)\qquad & \wmin\left\{ F(x)-L(x): x\in C,\; G_u(x)\in -S,\;\forall u\in\mathcal{U}\right\} .
\end{align}
Then, the {\it Lagrangian robust dual problem} and the {\it weak Lagrangian robust dual problem}  of $(\mathrm{RVP}^L)$ are,  respectively, 
\begin{align*}
	&({\rm RVD}^L)\quad\  \mathop{\wsup}\limits_{(T, u)\in    \L_+(S,K)\times \U} 
		\mathop{\winf}_{x\in C} (F-L+ T\circ G_u)(x)   ,\\
	&({\rm RVD}_w^L) \quad \mathop{\wsup}\limits_{(T,u)\in    \L^w_+(S,K)\times \U}  
		\mathop{\winf}_{(x,s)\in C\times S} \big[(F-L+ T\circ G_u)(x) + T(s)\big].
\end{align*}

Let $\emptyset\ne\V\subset \L(X,Y)$. We say that the \emph{robust strong $\V$-stable duality holds for the pair $({\rm RVP})-({\rm RVD})$}  if, for any $L\in \V$,
\begin{equation}
	\label{eq_dual}
	\wmax ({\rm RVD}^L)  = \winf ({\rm RVP}^L).
\end{equation}
When $\V = \L (X, Y)$ we will say that the \emph{robust strong  stable duality holds for the pair $({\rm RVP})-({\rm RVD})$} instead of ``the \emph{robust strong $\L(X,Y)$-stable duality holds for the pair $({\rm RVP})-({\rm RVD})$}''. 
It is obviously  that when $\V=\{0_\L\}$, the concept ``robust strong $\V$-stable duality'' reduces to the concept ``robust strong duality''. For the pair $({\rm RVP})-({\rm RVD}_w)$, the corresponding concepts (for instance, \emph{robust strong $\V$-stable duality holds for the pair $({\rm RVP})-({\rm RVD}_w)$}) will  be defined in the same way.  

We first introduce  the following {principles} of robust strong $\V$-stable duality.

\begin{theorem}[Principles of  robust  strong $\V$-stable duality I] 
	\label{thm_SD1} 
	Consider the following statements:
	\begin{itemize}[nosep]
		\item[$({\rm e})$] $\epi(F+I_A)^*\cap (\V\times Y)= \A \cap (\V\times Y)$, 

		\item[$({\rm f})$] $\epi(F+I_A)^*\cap (\V\times Y)= \B \cap (\V\times Y)$, 

		\item[$({\rm g})$] The robust strong $\V$-stable duality holds for the pair ${\rm (RVP)-(RVD)},$

		\item[$({\rm h})$] The robust strong $\V$-stable duality holds for the pair $({\rm RVP})-({\rm RVD}_w).$
	\end{itemize}
	Then, one has $[({\rm e})\Longleftrightarrow({\rm g})]$ and  $[({\rm f})\Longleftrightarrow({\rm h})]$.
\end{theorem}

\begin{proof}
 {\it Proof of $[({\rm e}) \Longrightarrow ({\rm g})]$.} Take $L\in \V$, we will prove that $\winf ({\rm RVP}^L)=\wmax ({\rm RVD}^L)$.
Firstly, it is worth noting that the problems $({\rm RVP}^L)$ and $ ({\rm RVD}^L)$ are respectively nothing else but the problems  $({\rm RVP})$ and $ ({\rm RVD})$ with $F$ replaced by $F-L$. As $({\rm e})$ holds, according to Theorem \ref{thm_CSFI}, it holds,  for all $y\in Y$, 
\begin{gather*}
\Big(G_u(x)\in -S,\;  x\in C, \; \forall u\in \mathcal{U}\; 
				\Longrightarrow \; y-L(x)+F(x)\notin -\inte K\Big)\\
\Updownarrow\\
\Big(\exists u\in\U,\; \exists T\in \L_+(S,K): F(x)+T\circ G_u(x)-L(x)+y\notin -\inte K,\;\forall x\in C\Big), 
\end{gather*}
or in  other words, \cite[Theorem 1 $(ii)$]{DL17VJM} holds with $F-L$ replacing $F$. Repeat whole the proof of $[({\rm a}) \Longrightarrow ({\rm b})]$ of \cite[Theorem 6]{DL17VJM}, we obtain $\winf ({\rm RVP}^L)=\wmax ({\rm RVD}^L)$.

 {\it Proof of $[({\rm e}) \Longleftarrow ({\rm g})]$.}
It follows from Proposition \ref{epi_include} that
\begin{equation}\label{eq_6.4bisbis}
 \A\cap (\V\times \W)\subset \epi (F+I_A)^*\cap (\V\times \W).
\end{equation}
So, to prove $({\rm e})$ holds, it is sufficient to check that the converse  inclusion of \eqref{eq_6.4bisbis} holds.
Take $(L,y)\in \epi (F+I_A)^*\cap (\V\times \W)$. Then, according to \eqref{eq111},
$$y+F(x)-L(x)\notin -\inte K,\; \forall x\in A\cap \dom F.$$
Use the same argument as in the proof of (51) of \cite{DL17VJM}  (page 312) with $F-L$ replacing $F$, one gets the existence of $(T,u)\in \L_+(S,K)\times \U$ such that $(L,y)\in \epi (F+I_C+T\circ G_u)^*\subset \A$. So, the  converse   inclusion of \eqref{eq_6.4bisbis} holds. 

 {\it Proof of $[({\rm f}) \Longleftrightarrow ({\rm h})]$.} The proof is similar to the one of $[({\rm e}) \Longleftrightarrow ({\rm g})]$. 
 \end{proof}

We now turn to  the convex case i.e., the case where $(\mathcal{H}_0)$ holds. In such a case, with the help of the results  established in Section 4, the qualifying conditions in (e) and (f) can be described in terms of sectional convexity and closedness.

\begin{theorem}[Principles of  convex robust  strong $\V$-stable duality I]
	\label{thm_SD2}
Assume that $(\mathcal{H}_0)$ holds. 	Consider the following statements:
 	\begin{itemize}[nosep]
		\item[$({\rm e}_1)$] $ \exists k\in \inte K$ s.t.  $\A$ is $k$-sectionally 
			convex and closed regarding $\V\times Y$,

		\item[$({\rm f}_1)$]   $\exists k\in \inte K$ s.t.  $\B$ is $k$-sectionally 
			convex  and closed regarding $\V\times Y$.
	\end{itemize}
	Then, one has $[({\rm e}_1)\Longleftrightarrow({\rm g})]$ and  $[({\rm f}_1)\Longleftrightarrow({\rm h})]$, {where $({\rm g})$ and $({\rm h})$ are in Theorem \ref{thm_SD1}}. 
\end{theorem}

\begin{proof}  As $(\mathcal{H}_0)$ holds, it follows from Theorem 
	\ref{thm_epi-presenting2} that 
	$$ \epi (F+I_A)^*= \cl(\sco\nolimits_k\A)=    \cl(\sco\nolimits_k\B),\; \forall k\in \inte K $$
	So, the statements $({\rm e}_1)$ and $({\rm f}_1)$ are equivalent to
		$\epi(F+I_A)^{\ast }\cap (\V\times Y)=\A\cap (\V\times Y)$
 and  
$\epi(F+I_A)^{\ast }\cap (\V\times Y)=\B\cap (\V\times Y)$,
respectively.	The conclusion now follows from Theorem \ref{thm_SD1}.  
\end{proof}

\begin{theorem}[Convex robust  strong $\V$-stable duality]
	\label{thm_SD2bb}
Assume that $(\mathcal{H}_0)$ and   the following condition hold:
 	\begin{itemize}[nosep]
		\item[$({\rm e}_2)$]  $ \exists k\in \inte K$ s.t.   {$\A_k$} is $k$-sectionally 
			convex  and  closed regarding $\V\times Y$.

	\end{itemize}
	Then, strong robust $\V$-stable duality holds for   ${\rm (RVP)-(RVD)}$ and $({\rm RVP})-({\rm RVD}_w).$
\end{theorem}

\begin{proof}  Use the same argument as in the proof of Theorem \ref{thm_5.3} (with $\W=Y$)
we can show that if  $({\rm e}_2)$ holds then $(\rm e)$ and $(\rm f)$ in Theorem \ref{thm_SD1} hold  and then, 
the conclusion follows from  Theorem \ref{thm_SD1}.
\end{proof}

\begin{theorem}[Convex robust  strong stable duality I]
\label{thm_SD2bis}
Assume that $\U$ is a compact space, that $Z$ is a normed space, and that  the  hypotheses $(\mathcal{H}_0)$, $(\mathcal{H}_1)$, $(\mathcal{H}_2)$,  and  the Slater-type condition $(C_0)$ hold. Then  robust strong stable duality holds for two pairs  ${\rm (RVP)-(RVD)}$ and $({\rm RVP})-({\rm RVD}_w).$
\end{theorem}

\begin{proof} It follows from  Proposition \ref{cor_5.1zzz} and Theorem \ref{thm_SD2}.  
\end{proof}

\begin{remark}
It worth observing  that the hypotheses $(\mathcal{H}_1)$, $(\mathcal{H}_2)$ and the condition $(C_0)$ do not concern  the objective mapping $F$. So, the conclusion of Theorem \ref{thm_SD2bis} until holds true when $F$ is replaced by arbitrary proper $K$-convex and positively $K$-lsc. In  other words, under the assumptions of Theorem \ref{thm_SD2bis}, the robust strong duality for pairs  ${\rm (RVP)-(RVD)}$ and $({\rm RVP})-({\rm RVD}_w)$ are stable in a stronger sense that the objective mapping $F$ can be perturbed by arbitrary mapping provided that properties: ``proper", ``$K$-convex", and ``positively $K$-lsc"  are still reserved.

\end{remark}

\subsection{Robust duality via $k$-sectional convexity}
{Fix} $k\in \inte K$. By letting {$\A_k$} play the role of $\A$ (or $\B$) as the qualifying set, one {gets} the {\it dual problem $({\rm RVD}^k)$} as follows:
\begin{align*}
	&{({\rm RVD}^k)}\quad\  \mathop{\wsup}\limits_{(z^*, u)\in    S^+\times \U} 
		\mathop{\winf}_{x\in C} [F(x)+ (z^*\circ G_u)(x)k]. 
\end{align*}
The {\it robust strong duality} and the {\it $\V$-stable robust strong duality} for pair $({\rm RVP})-({\rm RVD}^k)$ can be understood by the same way   as the previous subsection.

\begin{theorem}[Principles of  robust  strong $\V$-stable duality II] 
	\label{thm_SD1b} 
	Let $\emptyset \ne \V\subset Y$. The following statements are equivalent:
	\begin{itemize}[nosep]
		\item[$({\rm e}_k)$] $\epi(F+I_A)^*\cap (\V\times Y)= {\A_k} \cap (\V\times Y)$, 

		\item[$({\rm g}_k)$] The robust strong $\V$-stable duality holds for the pair ${\rm (RVP)}-({\rm RVD}^k).$
	\end{itemize}
	\end{theorem}

\begin{proof}
Use the same argument as in proof of  $[({\rm e})\Longleftrightarrow({\rm g})]$ in Theorem \ref{thm_SD1}, using  Theorem \ref{thm_CSFIbbis}  instead of Theorem \ref{thm_CSFI}.\end{proof}

\begin{theorem}[Principle of convex  robust strong $\V$-stable duality II]
	\label{thm_SD2b}
	Let $\emptyset \ne \V\subset Y$.
	Assume $(\mathcal{H}_0)$ hold and consider the following statement: 
	  	\begin{itemize}[nosep]
		\item[$({\rm e}^\prime_k)$] {$\A_k$} is $k$-sectionally 
			convex  and  closed regarding $\V\times Y$.
	\end{itemize}
	Then, $[({\rm e}^\prime_k)\Longleftrightarrow ({\rm g}_k)]$, where $({\rm g}_k)$ is the statement in Theorem \ref{thm_SD1b}.   
\end{theorem}
\begin{proof}
As  $(\mathcal{H}_0)$ is satisfied,    Theorem \ref{thm_epi_presenting1} gives    $\epi (F+I_A)^*= {\cl(\sco_{k} \A_{k})}$ and  so,   $({\rm e}_k^\prime)$  in this case is nothing else but $({\rm e}_k)$ in Theorem \ref{thm_SD1b}.  The conclusion now follows from Theorem \ref{thm_SD1b}.\end{proof}

\begin{theorem}[Convex robust strong stable duality II]
\label{thm_SD2bisb}
Assume that $\U$ is a compact space, that $Z$ is a normed space, and that  the  hypotheses $(\mathcal{H}_0)$, $(\mathcal{H}_1)$, $(\mathcal{H}_2)$ and  the Slater-type condition $(C_0)$ hold. Then  strong robust stable duality holds for pair  $({\rm RVP})-({\rm RVD}^k).$
\end{theorem}

\begin{proof}
It follows from Proposition \ref{cor_5.1zzz}   and Theorem \ref{thm_SD2b}.
\end{proof}

\section{Applications  to  robust convex optimizations} 
In this section we will specialize  our results on   robust  strong (stable)  duality for vector problems obtained in Section 6 to some classes of  (scalar) robust convex optimizations, which means that  we will consider the case  when $Y=\mathbb{R}$ and $K=\mathbb{R}_+$ (and hence, $\R^\bullet \equiv \overline{\R} := \R \cup \{ \pm \infty\}$). In this setting, we will write    $f$ (instead of $F$)  for the objective function of problems. Observed also that in this case  $\L(X, Y)$ becomes $X^\ast$, both the cones $\L_+(S, K)$ and $\L_+^w(S, K)$ now collapse to the positive dual cone $S^+$ of $S$,  and  the conjugate $f^*$ is none  other than the usual conjugate $f^*$ in the sense of convex analysis.  As results, the specification even to robust scalar problems still produce some  new   robust strong duality results, some  that  extend, or cover the known ones in the literature.

\subsection{General  robust convex optimization problem}

Consider the  robust convex optimization problem: 
\begin{align}
%\label{RP}
	{\rm(RP)}\qquad {\inf  }\left\{ f(x): x\in C,\; G_u(x)\in -S,\; \forall u\in \U\right\} \nonumber
\end{align}
where $X,Z$ are lcHtvs,  $S$ is a closed convex cone of $Z$,  $\U$ is an uncertainty set, $f\in \Gamma(X)$,  $G_u \colon X\to Z^\bullet$ is proper, $S$-convex and $S$-epi closed mapping for all $u\in \U$,  and $C\subset X$ is  a nonempty closed and convex subset of $X$. Note that under these assumptions, $(\mathcal{H}_0)$  is satisfied.

 Let us retain call   $ A$ (in \eqref{feas-A})   the \textit{feasible set} of ${\rm (RP)}$. Assume   that $ A\cap \dom f \not=\emptyset$. 

For the problem (RP),  the qualifying   sets $\A$,     $\B$,  and $\A_k$ (for any $k \in \inte \R_+$) in Section 4  collapse to the unique one    
	\begin{align*}
	\widehat{\A}  &:= \bigcup\limits_{(z^*, u )\in  S^+ \times \U } \epi (f+i_{C}+z^*\circ G_u)^{\ast }.
	\end{align*}
The Lagrangian dual problem $({\rm RVD})$, the weak Lagrangian dual problem $({\rm RVD}_w)$ and dual problem $({\rm RVD}^k)$  in this case  collapse to the    unique Lagrange dual problem  $({\rm RD})$ of $\rm(RP)$: 
\begin{align*}
	({\rm RD})\qquad \mathop{{\sup}}\limits_{( z^*, u )\in  S^+ \times \U}  
		\mathop{\inf}_{x\in C} (f+ z^*\circ G_u)(x)
\end{align*}
and, for all $\emptyset\ne \V\subset X^*$, ``the robust strong $\V$-stable duality holds for the pair $({\rm RP})-({\rm RD})$'' means that, for all $x^*\in \V$, 
$$ {\inf  }\left\{ f(x)-\la x^*,x\ra: x\in C,\; G_u(x)\in -S,\; \forall u\in \U\right\} = \mathop{{\max}}\limits_{( {z^*}, u )\in  S^+ \times \U}  
		\mathop{\inf}_{x\in C} (f-x^*+{z^*\circ G_u})(x). $$
The  next two corollaries come directly from  Theorem  \ref{thm_SD2}  and Theorem 
\ref{thm_SD2bis}, respectively.

\begin{corollary}{\rm \cite[Theorem 6.3]{DMVV17-Robust-SIAM}}{\rm (Principle of   robust convex  strong $\V$-stable duality)}
	\label{cor_SD1} 
	Let $\emptyset\ne \V\subset X^*$. The following statements are equivalent:
	\begin{itemize}[nosep]
		\item[$(\rm k)$] The set $\widehat{\mathcal{A}}$ is  closed and convex regarding  $\V\times \R$,

		\item[$(\rm l)$] The robust strong $\V$-stable duality holds for the pair $({\rm RP})-({\rm RD}).$
	\end{itemize}
\end{corollary}

\begin{corollary}\label{cor_7.2}
Assume that $\U$ is a compact space, that $Z$ is {a normed space}, and that  the  hypotheses $(\mathcal{H}_1)$ and $(\mathcal{H}_2)$ {in Section 5} hold.
Assume further that the following condition holds:
\begin{itemize}[nosep]
\item[$(\widehat{\rm C}_0)$] {$\forall u\in \U,\; \exists x_u\in C\cap \dom f: G_u(x_u)\in -\inte S.$}
\end{itemize}
Then, the robust strong $\V$-stable duality holds for the pair $({\rm RP})-({\rm RD}).$
\end{corollary}

\subsection{Robust  convex programming under uncertain inequality constraints}

Consider the {\it robust convex  programming} of the form 
\begin{equation*}
	{\rm({RCP})}\qquad   \inf\left\{ f(x): x\in C,\; g_t(x,u_t)\le 0,\; \forall u_t\in U_t,\;\forall t\in T\right\}
\end{equation*} 
where $f\in\Gamma(X)$, $T$ is a possibly infinite index set, $U_t$ is uncertainty set for each $t\in T$,
$g_t(.,u_t)\in \Gamma(X)$  for all $u_t\in U_t$ and $t\in T$, and $C\subset X$ is nonempty closed and convex.
Let $A:=\left\{x\in C:  g_t(x,u_t)\le 0,\; \forall u_t\in U_t,\; \forall t\in T 
\right\}$ and assume  that $ A\cap\dom f\ne\emptyset$.

We will propose  several  ways   to  transform $({\rm RCP})$ to the form of $({\rm RP})$.  The robust strong (stable) duality results in the  previous subsection are then applied  to get the variants of robust strong duality results for  (RCP), which are new, extend or cover the known ones in the literature. 

$\bullet$ {\it The first way:} Take $Z=\mathbb{R}^{T}$, $S=\mathbb{R}^T_+$ , $\U=\prod_{t\in T} U_t$, $G_u(x)=(g_t(x,u_t))_{t\in T}$ for all $x\in X$ and $u=(u_t)_{t\in T} \in \U$.
We consider  $\mathbb{R}^{T}$  endowed with the
product topology and its dual space, $\mathbb{R}^{(T)}$, is the space 
 of generalized finite sequences (i.e., the functions   $\lambda
=(\lambda _{t})_{t\in T}\in \mathbb{R}^{T}$ such that its supporting
set $\supp\lambda :=\{t\in T:\lambda _{t}\neq 0\}$ is finite)
 with dual
product defined by 
\begin{equation*}
\left\langle \lambda ,v\right\rangle :=\left\{ 
\begin{array}{ll}
\sum\limits_{t\in \supp\lambda }\lambda _{t}v_{t}, & \text{if }\lambda
\neq 0_{T}, \\ 
\ \ 0, & \text{otherwise,}%
\end{array}%
\right. \text{ }
\end{equation*}%
for all  $\left( \lambda ,v\right) \in \mathbb{R}^{(T)}\times \mathbb{R
}^{T}.$ The positive cones in  $\mathbb{
R}^{T}$ and in $\mathbb{%
R}^{(T)}$ is denoted by $\mathbb{R}_+^T$ and  $\mathbb{R}_{+}^{(T)}$, respectively.  In this setting, the qualifying set $\widehat{\mathcal{A}}$ becomes
\begin{align*}
\widehat\A_1&:= \bigcup\limits_{\substack{(\lambda_t)_{t\in T}\in \mathbb{R}_+^{(T)}\\(u_t)_{t\in T}\times \U} } \epi \left(f+ i_C +\sum_{t\in T} \lambda_tg_t(.,u_t)\right)^{\ast } 
\end{align*}
%Moreover, as the functions $f, g_i(.,u_i)\colon \mathbb{R}^n\to \mathbb{R}$ is convex (for all $u_i\in U_i$, $i=1,\ldots, m$), they are continuous on $\mathbb{R}^n$, and hence,
%\begin{align*}
%\widetilde\A&= \epi f^*+\bigcup\limits_{\substack{(\lambda_i,u_i)\in\mathbb{R}_+\times U_i\\ i=1,\ldots,m} }  \epi \left( \sum_{i=1}^m \lambda_ig_i(.,u_i)\right)^{\ast }.
%\end{align*}
and the robust dual problem $(\rm RD)$ now becomes 
$$({\rm RCD}_1)\qquad 
\mathop{{\sup}}\limits_{\substack{(\lambda_t)_{t\in T}\in \mathbb{R}_+^{(T)}\\(u_t)_{t\in T}\times \U} }
		\mathop{\inf}_{x\in C} \left(f(x)+ \sum_{t\in T}\lambda_t g_t(x,u_t)\right). 
$$
The robust  dual problem of this form   was considered in other works as \cite{DL17VJM,  DGLV17-Robust, JL10, JLW13}.
The next corollary is   a direct consequence of Corollary \ref{cor_SD1} which turns back to   \cite[Theorem 6.4]{DMVV17-Robust-SIAM} and  covers \cite[Theorem 4.1]{DGLV17-Robust} (for $i=O$) and  \cite[Theorem 3.1]{JL10}.

\begin{corollary} \label{cor_7.5}
	Let $\emptyset\ne \V\subset X^*$. The following statements are equivalent:
	\begin{itemize}[nosep]
		\item[$({\rm k}_1)$] The set $\widehat{\mathcal{A}}_1$ is  closed and convex regarding  $\V\times \R$,

		\item[$({\rm l}_1)$] The robust strong $\V$-stable duality holds for $({\rm RCP})-({\rm RCD}_1)$, i.e., for all $x^*\in\V$,
\begin{equation}\label{eq_7.6}
\hskip-1cm
\inf_{\substack{x\in C,\\g_t(x,u_t)\le 0,\; \forall u_t\in U_t, \; \forall t\in T}}\hskip-1.3cm  [f(x)-\la x^*,x\ra]=\mathop{{\max}}\limits_{\substack{(\lambda_t)_{t\in T}\in \mathbb{R}_+^{(T)}\\(u_t)_{t\in T}\times \U} }
		\mathop{\inf}_{x\in C} \left(f(x)-\la x^*,x\ra+ \sum_{t\in T}\lambda_t g_t(x,u_t)\right).
\end{equation}
	\end{itemize}
\end{corollary}

The next result  is a consequence of Corollary \ref{cor_7.2} and extends   \cite[Corollary 3.3]{JL10}.

\begin{corollary}
\label{cor_7.4zz}
Assume that $T$ is finite and that $U_t$ is a compact and convex subset of some  topological vector space for all $t\in T$, and $g_t(x,.)\in -\Gamma(U_t)$ for all $x\in C\cap\dom f$. Assume further that the following condition holds:
\begin{itemize}[nosep]
\item[$(\widehat{\rm C}_0^1)$] $\forall u=(u_t)_{t\in T}\in \U,\; \exists x_u\in C\cap \dom f: g_t(x_u,u_t)<0, \;\forall t\in T.$
\end{itemize}
Then, the robust strong stable duality holds for the pair $({\rm RCP})-({\rm RCD}_1)$. 
\end{corollary}

\begin{proof}
Firstly, $\U$ is compact (as $U_t$ is compact for all $t\in T$) and $Z:=\mathbb{R}^T$ is a normed space (note  that $T$ is finite).

 $\bullet$   We now prove that the collection $\Big((u_t)_{t\in T}^m\mapsto (g_t(x,u_t))_{t\in T}\Big)_{x\in C\cap\dom f}$ is uniformly $\mathbb{R}^{(T)}_+$-concave, or equivalently, the hypothesis $(\mathcal{H}_1)$ holds.
For this, take $(\lambda_t^j)_{t\in T} \in \mathbb{R}^{(T)}_+$ and  $(u_t^j)_{t\in T} \in \U$ ($j=1,2$), we will find $(\bar \lambda_t)_{t\in T}\in \mathbb{R}^{(T)}_+$ and $(\bar u_t)_{t\in T}\in \U$ such that
\begin{equation}\label{eq_7.3}
\sum_{t\in T}\lambda_t^1 g_t(x,u^1_t) + \sum_{t\in T}\lambda_t^2 g_t(x,u^2_t) \le \sum_{t\in T}\bar \lambda_t 
g_t(x,\bar u_t),\quad \forall x\in C\cap \dom f.
\end{equation}
To do this, for all $t\in T$, take  $\bar \lambda_t:=\lambda_t^1+\lambda_t^2$ and 
$$\bar u_t:=\begin{cases}
\frac{\lambda^1_t}{\lambda^1_t+\lambda_i^2}u_t^1+ \frac{\lambda^2_t}{\lambda^1_t+\lambda_t^2}u_t^2,&\textrm{if } \lambda_t^1>0 \textrm{ or }\lambda_t^2>0\\
u_t^1&\textrm{else.}
\end{cases}$$
As $g_t(x,.)$ is concave on the convex set $U_t$ for each $t \in T$, one has (see Example \ref{ex5.1}), for all $x\in C\cap\dom f$,
$$\lambda_t^1 g_t(x,u_t^1)+\lambda_t^2 g_t(x,u_t^2)\le \bar\lambda_t g_t(x, \bar u_t), $$
which, in turn,  yields  
 \eqref{eq_7.3}.

$\bullet$ Next,  for all $x\in C\cap\dom f$, as $g_t(x,.)\colon U_t\subset \mathbb{R}^{q_i}\to \mathbb{R}$ is usc for all $t\in T$, by Lemma  \ref{rem_5.1zzz}(iii)),   $(g_t(x,.))_{t\in T}$ is   $\mathbb{R}^{(T)}_+$-uniformly usc,  meaning  that $(\mathcal{H}_2)$ holds.

$\bullet$ Finally, the fulfilment of $(\widehat C_0^1)$ entails that the Slater-type condition $(\widehat C_0)$ in Corollary \ref{cor_7.2} holds.  The conclusion now follows from Corollary \ref{cor_7.2}.
\end{proof}

$\bullet$ {\it The second way:} Take $Z=\mathbb{R}$, $\U=T$, and
$G_t(x)=\sup_{v\in U_t}g_t(x,v)$ for all $x\in X$ and $t\in T$. Then, the qualifying set $\widehat{\mathcal{A}}$ becomes
\begin{align*}
\widehat\A_2   &= \bigcup\limits_{\substack{\lambda\ge 0,\; t\in T  }} \cl\co\left(\bigcup_{v\in U_t}\epi \left(f+i_C+\lambda g_t(.,v)\right)^{\ast }\right)
\end{align*}
(see \cite[Lemma 2.2]{GLi}).
%Moreover, as the functions $f, g_i(.,u_i)\colon \mathbb{R}^n\to \mathbb{R}$ is convex (for all $u_i\in U_i$, $i=1,\ldots, m$), they are continuous on $\mathbb{R}^n$, and hence,
%\begin{align*}
%\widetilde\A&= \epi f^*+\bigcup\limits_{\substack{(\lambda_i,u_i)\in\mathbb{R}_+\times U_i\\ i=1,\ldots,m} }  \epi \left( \sum_{i=1}^m \lambda_ig_i(.,u_i)\right)^{\ast }.
%\end{align*}
The robust dual problem $(\rm RD)$ now reduces to
$$({\rm RCD}_2)\qquad 
\mathop{{\sup}}\limits_{\substack{\lambda\ge 0,\  {t\in T}  }}
		\mathop{\inf}_{x\in C}\sup_{v\in U_t} \left(f(x)+\lambda g_t(x,v)\right).
$$
This form of robust dual problem of $({\rm RCIP})$ is proposed in \cite[Remark 10]{DGLV17-Robust}. As consequences of Corollaries \ref{cor_SD1}-\ref{cor_7.2}, one gets. 
\begin{corollary}\label{cor_7.5bb}
	Let $\emptyset\ne \V\subset X^*$. The following statements are equivalent:
	\begin{itemize}[nosep]
		\item[$({\rm k}_2)$] The set $\widehat{\mathcal{A}}_2$ is  closed and convex regarding  $\V\times \R$,
		\item[$({\rm l}_2)$] The robust strong $\V$-stable duality holds for the pair $({\rm RCP})-({\rm RCD}_2)$.  
	\end{itemize}
\end{corollary}

\begin{corollary}\label{cor_7.6zab}
Assume that T is a compact and convex subset of some topological vector space, and that the function $t\mapsto \sup_{v\in U_t}g_t(x, v)$ is concave and usc on $T$ for all $x\in C\cap\dom f$. Assume further that the following
 condition holds:
\begin{itemize}[nosep]
\item[$(\widehat{\rm C}_0^2)$] $\forall t\in T,\; \exists x_t\in C\cap \dom f: \sup_{v\in U_t} g_t(x_u,v)<0.$
\end{itemize}
Then, the robust strong stable duality holds for the pair $({\rm RCP})-({\rm RCD}_2)$.  
\end{corollary}

$\bullet$ {\it The third way:} Take $Z=\mathbb{R}$, $\U=\prod_{t\in T} U_t$, and
$G_{u}(x)=\sup_{t\in T}g_t(x,u_t)$ for all $x\in X$ and $u=(u_t)_{t\in T}$. Then, the qualifying set $\widehat{\mathcal{A}}$ becomes
\begin{align*}
\widehat\A_3
&= \bigcup\limits_{\substack{\lambda\ge 0,\\(u_t)_{t\in T} \in \U }} \cl\co\left(\bigcup_{t\in T}\epi \left(f+i_C+\lambda g_t(.,u_t)\right)^{\ast }\right)
\end{align*}
%(note that $A\ne\emptyset$, see \cite[Lemma 2.2]{LiNg}).
%Moreover, as the functions $f, g_i(.,u_i)\colon \mathbb{R}^n\to \mathbb{R}$ is convex (for all $u_i\in U_i$, $i=1,\ldots, m$), they are continuous on $\mathbb{R}^n$, and hence,
%\begin{align*}
%\widetilde\A&= \epi f^*+\bigcup\limits_{\substack{(\lambda_i,u_i)\in\mathbb{R}_+\times U_i\\ i=1,\ldots,m} }  \epi \left( \sum_{i=1}^m \lambda_ig_i(.,u_i)\right)^{\ast }.
and the robust dual problem $(\rm RD)$    of  of $(\rm RCP)$   turns to  new form   as follows:
$$({\rm RCD}_3)\qquad 
\mathop{{\sup}}\limits_{\substack{\lambda\ge 0,\\(u_t)_{t\in T} \in \U }}
		\mathop{\inf}_{x\in C}\sup_{t\in T} \left(f(x)+\lambda g_t(x,u_t)\right).
$$
Now, Corollaries \ref{cor_SD1}-\ref{cor_7.2} gives us the next results. 
\begin{corollary}\label{cor_7.5b}
	Let $\emptyset\ne \V\subset X^*$. The following statements are equivalent:
	\begin{itemize}[nosep]
		\item[$({\rm k}_3)$] The set $\widehat{\mathcal{A}}_3$ is  closed and convex regarding  $\V\times \R$,

		\item[$({\rm l}_3)$] The robust strong $\V$-stable duality holds for the pair $({\rm RCP})-({\rm RCD}_3)$. 	\end{itemize}
\end{corollary}

\begin{corollary}
Assume that  $U_t$ is a compact and convex subset of some vector topological space for all $t\in T$, and that the function $(u_t)_{t\in T}\mapsto \sup_{t\in T}g_t(x, u_t)$ is concave and usc on $\U$ for all $x\in C\cap\dom f$. Assume further that the following
 condition holds:
\begin{itemize}[nosep]
\item[$(\widehat{\rm C}_0^3)$] $\forall u=(u_t)_{t\in T}\in \U,\; \exists x_u\in C\cap \dom f: \sup_{t\in T} g_t(x_u,u_t)<0.$
\end{itemize}
Then, the robust strong stable duality holds for the pair $({\rm RCP})-({\rm RCD}_3)$. \end{corollary}

\begin{remark}
Noting  that there are still other  ways of transforming $({\rm RCP})$ to the form of $({\rm RP})$.  For instance, take $\mathfrak{U}=\{(t,u_t): t\in T, \; u_t\in U_t\}$, $Z=\mathbb{R}^\mathfrak{U}$, $S=\mathbb{R}^\mathfrak{U}_+$, $\U=\{\mathfrak{U}\}$ and $G_\mathfrak{U}=(g_t(.,u_t))_{(t,u_t)\in \mathfrak{U}}$. Then,  $\widehat\A$ and the dual problem $({\rm RD})$ become, respectively 
\begin{gather*}
\widehat{\A}_4=\bigcup_{\lambda\in \mathbb{R}^{(\mathfrak{U})}_+}\epi \left(f+i_C+\sum_{(t,u_t)\in \mathfrak{U}} \lambda_{(t,u_t)}g_t(.,u_t)\right)^*,\\
({\rm RCD}_4)\quad \sup_{\lambda\in \mathbb{R}^{(\mathfrak{U})}_+} \inf_{x\in C} \left( f(x)+\sum_{(t,u_t)\in \mathfrak{U}} \lambda_{(t,u_t)}g_t(x,u_t)\right)^*. 
\end{gather*}
 By  Proposition \ref{epirco2},  $\widehat \A_4$ is a convex subset of $X^*\times \mathbb{R}$. So, Corollary \ref{cor_SD1} yields the equivalence of two following assertions:
\begin{itemize}[nosep]
		\item[$({\rm k}_4)$] The set $\widehat{\mathcal{A}}_4$ is  closed regarding  $\V\times \R$, 
		\item[$({\rm l}_4)$] The robust strong $\V$-stable duality holds for the pair $({\rm RCP})-({\rm RCD}_4)$.
	\end{itemize}
This result covers \cite[Theorem 4.1]{DGLV17-Robust} for the case $i=C$. By using other suitable ways, we can get results that possibly cover \cite[Theorem 4.1]{DGLV17-Robust} with other values of  $i$. 
\end{remark}


\begin{thebibliography}{99}

 \bibitem{AliBurk}    Aliprantis, ChD, Burkinshaw, O.: Positive Operators. Academic Press, Orlando (1985)
 
 \bibitem{Barro} Barro, M.,  Ou\' edraogo, A., Traor\' e, S.: On Uncertain Conical Convex Optimization Problems. Pacific J. Optim.  {\bf 13}, 29-42 (2017)  

\bibitem{BB9} Beck, A.,  Ben-Tal, A.: Duality in robust optimization: Primal worst equals dual best. Oper. Res. Lett. {\bf 37}, 1-6 (2009)



\bibitem{BEN09} Ben-Tal, A., El Ghaoui, L.,  Nemirovski, A.: Robust Optimization. Princeton U.P., Princeton (2009)

\bibitem{Ber_survey} Bertsimas, D., Brown, D.B., Caramanis, C.: Theory
and applications of robust optimization. SIAM Rev. {\bf 53},  464-501 (2011)
 
    
    
\bibitem{Bot10} {Bo\c{t}}, R.I.: Conjugate Duality in Convex Optimization. Springer, Berlin (2010)


\bibitem{2} Bot, R.I., Grad, S.M.,  Wanka, G.: Duality in Vector Optimization. Springer-Verlag, Berlin, Germany (2009)



\bibitem{DGLM16} Dinh, N., Goberna, M.A., L\'{o}pez, M.A., Mo, T.H.: Farkas-type results for vector-valued functions with applications. J. Optim. Theory Appl. {\bf 173}, 357-390 (2017) 
\bibitem{Robust} Dinh, N.,   Goberna, M.A.,  L\'{o}pez, M.A., Mo, T.H.: Robust optimization revisited via robust vector Farkas lemmas. Optimization  {\bf 66}, 939-963 (2017)


\bibitem{DL17VJM} Dinh, N., Long, D.H.: Complete characterizations of robust strong duality for  robust vector optimization problems. Vietnam J. Math. {\bf 46}, 293-328 (2018) 


\bibitem{epi} Dinh, N.,  Goberna, M.A., Long, D.H., L\'opez, M.A.:  New Farkas-type results for vector-valued functions: a non-abstract approach. J. Optim. Theory Appl. (to appear). https://doi.org/10.1007/s10957-018-1352-z


\bibitem{DMVV17-Robust-SIAM} Dinh, N., Mo, T.H., Vallet, G., Volle, M.: A unified approach to robust Farkas-type results with applications to robust optimization problems. SIAM J. Optim. {\bf 27}, 1075-1101 (2017)

%\bibitem{HB} H. Brezis, \textit{Functional Analysis, Sobolev spaces and partial differential equations}, Springer, 2010.


\bibitem{DGLV17-Robust} Dinh, N., Goberna, M.A., Lopez, M.A., Volle, M.: A unifying approach to robust convex infinite optimization duality. J. Optim. Theory Appl. {\bf 174}, 650 - 685 (2017)


\bibitem{GJLL13} Goberna, M. A., Jeyakumar, V., Li, G., Lopez, M. A: Robust linear semi-infinite programming duality under uncertainty. Math. Program. {\bf 139}, Ser. B, 185-203  (2013) 


\bibitem{Gabrel}  Gabrel, V.,   Murat, C.,   Thiele, A.:  Recent advances in robust optimization: An overview. European J. Oper Res.,  
\textbf{235},   471-483 (2014) 



\bibitem{JL10} Jeyakumar, V., Li, G.Y.: Strong duality in robust convex programming: Complete characterizations. SIAM J. Optim. {\bf 20}, 3384-3407 (2010) 
\bibitem{JLW13} Jeyakumar, V., Li, G., Wang, J.H.: Some robust convex programs without a duality gap. J. Convex Anal. {\bf 20}, 377 - ��394 (2013)






\bibitem{GLi} Li, G., Ng. K.F.: On extension of Fenchel duality and its application.  SIAM J. Optim. {\bf 19}, 1489-1509 (2008)





\bibitem{Rudin91} Rudin, W.: Functional Analysis (2nd Edition). McGraw-Hill, N.Y. (1991)

%\bibitem{Tiel84} \textcolor{red}{J.V. Tiel, \textit{Convex Analysis - An Introductory Text}, John Wiley \& Sons Ltd., 1984.}


\bibitem{Tan92} Tanino, T.: Conjugate duality in vector optimization. J. Math. Anal. Appl. {\bf 167}, 84-97  (1992)  


\end{thebibliography}
\end{document}